\documentclass[reqno,11pt]{amsart}
\usepackage{amsmath, amsfonts,amsthm,amssymb,amsbsy,upref,color,graphicx,amscd,hyperref,enumerate,enumitem,relsize}
\usepackage{color}
\usepackage[active]{srcltx}
\usepackage[latin1]{inputenc}
\usepackage{bbm}
\usepackage{graphicx}
\usepackage{subfig} 
\usepackage{xfrac}
\usepackage{comment} 

\usepackage{dsfont}

\def\ds{\displaystyle}
\def\eps{{\varepsilon}}
\def\N{\mathbb{N}}
\def\O{\Omega}

\def\R{\mathbb{R}}

\newcommand{\be}{\begin{equation}}
\newcommand{\ee}{\end{equation}}

\newcommand{\ind}{\mathbbm{1}}

\theoremstyle{plain}
\newtheorem{teo}{Theorem}[section]
\newtheorem{lm}[teo]{Lemma}
\newtheorem{prop}[teo]{Proposition}
\newtheorem{coro}[teo]{Corollary}

\theoremstyle{definition}

\newtheorem{oss}[teo]{Remark}

\def\Xint#1{\mathchoice
   {\XXint\displaystyle\textstyle{#1}}%
   {\XXint\textstyle\scriptstyle{#1}}%
   {\XXint\scriptstyle\scriptscriptstyle{#1}}%
   {\XXint\scriptscriptstyle\scriptscriptstyle{#1}}%
   \!\int}
\def\XXint#1#2#3{{\setbox0=\hbox{$#1{#2#3}{\int}$}
     \vcenter{\hbox{$#2#3$}}\kern-.5\wd0}}

\def\aver#1{\Xint-_{#1}}

\DeclareMathOperator{\dive}{div}

\usepackage{refcount}

\newcounter{cte}

\numberwithin{equation}{section}

\begin{document}

\author{Baptiste Trey}

\title[Lipschitz continuity of the eigenfunctions]{Lipschitz continuity of the eigenfunctions on optimal sets for functionals with variable coefficients}

\begin{abstract}
	This paper is dedicated to the spectral optimization problem 
	\begin{equation*}
	\min \big\{ \lambda_1(\O)+\cdots+\lambda_k(\O) + \Lambda|\O| \ : \ \O \subset D \text{ quasi-open} \big\}
	\end{equation*}
	where $D\subset\R^d$ is a bounded open set and $0<\lambda_1(\O)\leq\cdots\leq\lambda_k(\O)$ are the first $k$ eigenvalues on $\O$ of an operator in divergence form with Dirichlet boundary condition and H\"{o}lder continuous coefficients.
	We prove that the first $k$ eigenfunctions on an optimal set for this problem are locally Lipschtiz continuous in $D$ and, as a consequence, that the optimal sets are open sets. We also prove the Lipschitz continuity of vector-valued functions that are almost-minimizers of a two-phase functional with variable coefficients.
\end{abstract}

\date{\today}
\maketitle
\tableofcontents

\section{Introduction and main results}

Let $D$ be a bounded open subset of $\R^d$ and $\Lambda$ be a positive constant.
We consider the spectral optimization problem 
\begin{equation}\label{e:shapeopt}
\min \big\{ \lambda_1(\O)+\cdots+\lambda_k(\O) + \Lambda|\O| \ : \ \O \subset D \text{ quasi-open} \big\}
\end{equation}
where $0<\lambda_1(\O)\leq\cdots\leq\lambda_k(\O)$ denote the first $k$ eigenvalues, counted with the due multiplicity, of the operator in divergence form $-b(x)^{-1}\text{div}\,(A_x\nabla\cdot)$. This means that for every $\lambda_i(\Omega)$ there is an eigenfunction $u_i\in H^1_0(\Omega)$ such that 
\begin{equation}\label{e:defoperator}
\left\{
      \begin{aligned}
        -\dive(A\nabla u_i) &= \lambda_i(\Omega) \,b\,u_i & &\text{ in } \O \\
        u_i&=0 & &\text{ on } \partial\O. \\
      \end{aligned}
    \right.
\end{equation}

The aim of the present paper is twofold. From one side, we prove a Lipschitz regularity result for vector-valued functions which are almost-minimizers of a two-phase functional with variable coefficients (Theorem \ref{t:main2}). On the other hand, we show that if $\O^\ast$ is an optimal set for \eqref{e:shapeopt}, then the vector $U=(u_1,\dots,u_k)$ of the first $k$ eigenfunctions on $\O^\ast$ satisfies the almost-minimality condition of Theorem \ref{t:main2}, and hence that the eigenfunctions $u_1,\dots,u_k$ are Lipschitz continuous. 

We first state our Lipschitz regularity result for  eigenfunctions on optimal sets for \eqref{e:shapeopt}. 

\begin{teo}\label{t:main1}
Let $D\subset\R^d$ be a bounded open set and let $\Lambda>0$.  Let $A: D\rightarrow\text{Sym}_d^+$ be a matrix valued function satisfying \eqref{e:holderA} and \eqref{e:ellipA} and let $b \in L^\infty(D)$ be a function satisfying \eqref{e:hypfctb} (see below). 
Then the spectral optimization problem \eqref{e:shapeopt} admits a solution $\Omega^\ast$. Moreover, the first $k$ eigenfunctions on any optimal set $\O^\ast$ are locally Lipschitz continuous in $D$. As a consequence, every optimal set for \eqref{e:shapeopt} is an open set. 
\end{teo}

In \cite{briancon-hayouni-pierre-05}, Briancon, Hayouni and Pierre proved the Lipschitz continuity of the first eigenfunction on an optimal set which minimizes the first eigenvalue of the Dirichlet Laplacian among all sets of prescribed volume included in a box. Their proof, which is inspired by the pioneering work of Alt and Caffarelli in \cite{alt-caffarelli-81} on the regularity for a free boundary problem, relies on the fact that the first eigenfunction is the minimum of a variational problem. 
For spectral optimization problems involving higher eigenvalues, the study of the regularity of the optimal sets and the corresponding eigenfunctions is more involved due to the variational characterization of the eigenvalue $\lambda_k$ through a min-max procedure. 
In \cite{bucur-mazzoleni-pratelli-velichkov-15} the authors considered the spectral functionals $F(\lambda_1(\O),\dots,\lambda_k(\O))$ which are bi-Lipschitz with respect to each eigenvalue $\lambda_i(\O)$ of the Dirichlet Laplacian, a typical example being the sum of the first $k$ eigenvalues. 
In particular, they proved the Lipschitz continuity of the eigenfunctions on optimal sets minimizing the sum $\lambda_1(\O)+\dots+\lambda_k(\O)$ among all shapes $\O\subset\R^d$ of prescribed measure (see \cite[Theorem 6.1]{bucur-mazzoleni-pratelli-velichkov-15}). The present paper extends this result to the case of an operator with variable coefficients, but with a completely different proof. 

Concerning spectral optimization problems involving an operator with variable coefficients, a regularity result has been obtained in \cite{russ-trey-velichkov-19}, where the authors consider the problem of minimizing the first eigenvalue of the operator with drift $-\Delta+\nabla\Phi\cdot\nabla$, $\Phi\in W^{1,\infty}(D,\R^d)$, under inclusion and volume constraints. We stress out that our result also applies to this operator with drift since it corresponds to the special case where $A=e^{-\Phi}\text{Id}$ and $b=e^{-\Phi}$. We would like also to mention a recent work of Lamboley and Sicbaldi in \cite{lamboley-sicbaldi-19} where they prove an existence and regularity result for Faber-Krahn minimizers in a Riemanninan setting.

Let us highlight that the Lipschitz regularity of the eigenfunctions in Theorem \ref{t:main1} turned out to be a quite difficult question due to both the min-max nature of the eigenvalues and the presence of the variable coefficients, but it is an important first step for the analysis of the regularity of the free boundary of the optimal shapes for \eqref{e:shapeopt} which we study in \cite{trey-20}.
\medskip
 
As already pointed out, the proof of Theorem \ref{t:main1} goes through the study of the Lipschitz regularity of vector-valued almost-minimisers for a two-phase functional with variable coefficients. 
Our approach is to reduce from the non-constant coefficients case to the constant coefficients-one by a change of variables and is inspired by \cite{spolaor-trey-velichkov-19}, where the authors prove free boundary regularity of almost-minimizers of the one-phase and two-phase functionals in dimension $2$ using an epiperimetric inequality. 
The second contribution which was a strong inspiration for our work is of David and Toro in \cite{david-toro-15}. They in particular prove the Lipschitz regularity of almost-minimizers of the one-phase and the two-phase functionals with constant coefficients (see also \cite{david-engelstein-toro-19} for free boundary regularity results).


We have the following result for almost-minimizers of the two-phase functional.

\begin{teo}\label{t:main2}
Let $D\subset \R^d$ be a bounded open set and let $\O\subset D$ be a quasi-open set. Let $A: D\rightarrow\text{Sym}_d^+$ be a matrix valued function satisfying \eqref{e:holderA} and \eqref{e:ellipA}. Let $f=(f_1,\dots,f_k) \in L^\infty(D,\R^k)$. Assume that $U=(u_1,\dots,u_k) \in H^1_0(\O,\R^k)$ is a vector-valued function such that
\begin{itemize}
\item $U$ is a solution of the equation
\begin{equation}\label{e:main2a}
-\dive(A\nabla U) = f \quad\text{in}\quad\O,
\end{equation}
\item $U$ satisfies the following quasi-minimality condition: for every $C_1>0$, there exist constants $\eps\in(0,1)$ and $C>0$ such that
\begin{equation}\label{e:main2b}
\int_DA\nabla U\cdot\nabla U + \Lambda|\{|U|>0\}| \leq \big(1+C\|U-\tilde{U}\|_{L^1}\big)\int_DA\nabla\tilde{U}\cdot\nabla\tilde{U} + \Lambda|\{|\tilde{U}|>0\}|,
\end{equation}
for every $\tilde{U}\in H^1_0(D,\R^k)$ such that $\|U-\tilde{U}\|_{L^1}\leq\eps$ and $\|\tilde{U}\|_{L^\infty}\leq C_1$.
\end{itemize}
Then the vector-valued function $U$ is locally Lipschitz continuous in $D$.
\end{teo}

\begin{oss}[On the assumption \eqref{e:main2b} of Theorem \ref{t:main2}]
The quasi-minimality in Theorem \ref{t:main2} is not local but naturally arises from the shape optimization problem \eqref{e:shapeopt} (see Proposition \ref{p:quasiminU}). We stress out that our conclusion also holds, with exactly the same proof, if the quasi-minimality property \eqref{e:main2b} is replaced by its "local" version, namely: for every $C_1>0$, there exist constants $r_0\in(0,1)$ and $C>0$ such that for every $x\in D$ and every $r\leq r_0$ such that $B_r(x)\subset D$ we have
\begin{equation*}
\int_{B_r(x)}A\nabla U\cdot\nabla U + \Lambda|\{|U|>0\}\cap B_r(x)| \leq \big(1+C\|U-\tilde{U}\|_{L^1}\big)\int_{B_r(x)}A\nabla\tilde{U}\cdot\nabla\tilde{U} + \Lambda|\{|\tilde{U}|>0\}\cap B_r(x)|,
\end{equation*}
for every $\tilde{U}\in H^1_0(D,\R^k)$ such that $U-\tilde{U}\in H^1_0(B_r(x),\R^k)$ and $\|\tilde{U}\|_{L^\infty}\leq C_1$.
\end{oss}

\begin{oss}
We point out that we will only use the assumption \eqref{e:main2a} to prove that $U$ is bounded and to get an almost-monotonicity formula (see Proposition \ref{p:monotony} and Corollary \ref{c:monotony}).

In \cite[Theorem 6.1]{david-toro-15}, David and Toro proved an almost-monotonicity formula for  quasi-minimizers in the case of the Laplacian. It is natural to expect that the same holds for an operator with variable coefficients, but we will not address this question in the present paper since we are mainly interested in the Lipschitz continuity of the eigenfunctions on optimal shapes for the problem \eqref{e:shapeopt} for which the equation \eqref{e:main2a} is already known.

However, soon before the present paper was published online, a new preprint of the same authors, in collaboration with Engelstein and Smit Vega Garcia (see \cite{david-engelstein-garcia-toro-19}), appeared on Arxiv. They prove a regularity result for functions satisfying a suitable quasi-minimality condition for operators with variable coefficients. 
We stress that the present paper and the work in \cite{david-engelstein-garcia-toro-19} were done in a completely independent way.
We notice that our main result neither directly implies nor is directly implied by the main result from \cite{david-engelstein-garcia-toro-19}.

\end{oss}

\noindent{\bf Notations.} 
Let us start by setting the assumptions on the coefficients of the operator that we will use throughout this paper.
The matrix-valued function $A=(a_{ij})_{ij} : D\rightarrow \text{Sym}_d^+$ has H\"{o}lder continuous coefficients and is uniformly elliptic, where $\text{Sym}_d^+$ denotes the family of all real positive symmetric $d \times d$ matrices. Precisely, there exist positive constants $\delta_{\text{\tiny\sc A}},c_{\text{\tiny\sc A}}>0$ and $\lambda_{\text{\tiny\sc A}}\geq 1$ such that
\begin{equation}\label{e:holderA}
|a_{ij}(x)-a_{ij}(y)|\le c_{\text{\tiny\sc A}}|x-y|^{\delta_{\text{\tiny\sc A}}},\quad\text{for every}\quad i,j\quad\text{and}\quad  x,y\in D\,;
\end{equation}
\begin{equation}\label{e:ellipA}
\frac{1}{\lambda_{\text{\tiny\sc A}}^2}|\xi|^2\le \xi\cdot A_x\,\xi=\sum_{i,j=1}^da_{ij}(x)\xi_i\xi_j\le \lambda_{\text{\tiny\sc A}}^2|\xi|^2,\quad\text{for every}\quad x\in D \quad\text{and}\quad \xi\in\R^d.
\end{equation}    
The function $b \in L^\infty(D)$ is positive and bounded away from zero: there exists $c_b>0$ such that
\begin{equation}\label{e:hypfctb}
c_b^{-1} \leq b(x) \leq c_b \quad \text{for almost every} \quad x\in D.
\end{equation}

We now fix some notations and conventions. For $x\in\R^d$ and $r>0$ we use the notation $B_r(x)$ to denote the ball centred at $x$ of radius $r$ and we simply write $B_r$ if $x=0$. We denote by $|\O|$ the Lebesgue measure of a generic set $\O\subset \R^d$ and by $\omega_d$ the Lebesgue measure of the unit ball $B_1\subset \R^d$. The $(d-1)$-dimensional Hausdorff measure is denoted by $\mathcal{H}^{d-1}$. Moreover, we define the positive and the negative parts of a function $u:\R\rightarrow\R$ by
\begin{equation*}
u^+=\max(u,0)\qquad\text{and}\qquad u^-=\max(-u,0).
\end{equation*}

For a quasi-open set $\O\in\R^d$ we denote by $H^1_0(\O)$ the Sobolev space defined as the set of functions $u\in H^1(\R^d)$ which, up to a set of capacity zero, vanishe outside $\O$; that is
\begin{equation*}
H^1_0(\O)=\{ u\in H^1(\R^d)\ :\ u=0\text{ quasi-everywhere in } \R^d\setminus\O \}.
\end{equation*}
(see e.g. \cite{henrot-pierre-05} for a definition of the capacity). Notice that if $\O$ is an open set, then $H^1_0(\O)$ is the usual Sobolev space defined as the closure of the smooth real-valued functions with support compact $C_c^\infty(\O)$ with respect to the norm $\|u\|_{H^1}=\|u\|_{L^2}+\|\nabla u\|_{L^2}$. We denote by $H^1_0(\O,\R^k)$ the space of vector-valued functions $U=(u_1,\dots,u_k):\O\to\R^k$ such that $u_i\in H^1_0(\O)$ for every $i=1,\dots,k$, and endowed with the norm
\begin{equation*}
\|U\|_{H^1(\O)}=\|U\|_{L^2(\O)}+\|\nabla U\|_{L^2(\O)}= \sum_{i=1}^k \big(\|u_i\|_{L^2(\O)} +\|\nabla u_i\|_{L^2(\O)} \big).
\end{equation*}
We also define the following norms (whenever it makes sense)
\begin{equation*}
\|U\|_{L^1(\O)}=\sum_{i=1}^k\|u_i\|_{L^1(\O)}\qquad\text{and}\qquad \|U\|_{L^\infty(\O)}=\sup_{1\leq i\leq k}\|u_i\|_{L^\infty(\O)}.
\end{equation*}
Moreover, for $U=(u_1,\dots,u_k):\O\to\R^k$ we set $|U|=u_1^2+\cdots+u_k^2$, $|\nabla U|^2=|\nabla u_1|^2+\cdots+|\nabla u_k|^2$ and $A\nabla U\cdot\nabla U=A\nabla u_1\cdot\nabla u_1+\cdots+A\nabla u_k\cdot\nabla u_k$. 
For $f=(f_1,\dots,f_k)\in L^2(\O,\R^k)$ we say that $U=(u_1,\dots,u_k)\in H^1_0(\O,\R^k)$ is solution to the equation 
\begin{equation*}
-\dive(A\nabla U)=f\quad\text{in}\quad\O,\qquad U\in H^1_0(\O,\R^k)
\end{equation*}
if, for every $i=1,\dots,k$, the component $u_i$ is solution to the equation
\begin{equation*}
-\dive(A\nabla u_i)=f_i\quad\text{in}\quad\O,\qquad u_i\in H^1_0(\O),
\end{equation*}
where the PDE is intended is the weak sense, that is
\begin{equation*}
\int_{\O}A\nabla u_i\cdot\nabla\varphi = \int_{\O}f_i\varphi\qquad\text{for every}\qquad\varphi\in H^1_0(\O).
\end{equation*}
Moreover, we always extend functions of the spaces $H^1_0(\O)$ and $H^1_0(\O,\R^k)$ by zero outside $\O$ so that we have the inclusions $H^1_0(\O)\subset H^1(\R^d)$ and $H^1_0(\O,\R^k)\subset H^1(\R^d,\R^k)$.

\section{Lipschitz continuity of quasi-minimizers}

This section is dedicated to the proof of Theorem \ref{t:main2}. Our approach is to locally freeze the coefficients to reduce to the case where $A=Id$.
More precisely, for every point $x\in D$, an almost-minimizer of the functional with variable coefficients becomes, in a new set of coordinates near $x$, an almost minimizer for a functional with constant coefficients. We stress out the dealing with the dependence of this change of variables with respect to the point $x$ is not a trivial task. We then adapt the strategy developed by David and Toro in \cite{david-toro-15} for almost-minimizers of a functional involving the Dirichlet energy. 

In this section, $u$ will stand for a coordinate function of the vector $U$ from Theorem \ref{t:main2}.
In subsection \ref{sub:contholdcont} we explicit the change of variables for which $u$ becomes a quasi-minimizer of the Dirichlet energy (in small balls of fixed center). We then prove that $u$ is continuous and we give an estimate of the modulus of continuity from which we deduce that $u$ is locally H\"{o}lder continuous in $D$. 

Subsection \ref{sub:boundlipcstustrpos} is addressed to the Lipschitz continuity of $u$ in some region where the function $u$ has a given sign. We show, using in particular the H\"{o}lder continuity of $u$, that most of the estimates proved in 
Subsection \ref{sub:contholdcont} can be improved provided that $u$ keeps the same sign. In this case, we prove that $u$ is Lipschitz continuous and we provide a bound on the Lipschitz constant of $u$. We also show that $u$ is $C^{1,\beta}$-regular for some $\beta\in(0,1)$. 
Next, we show that under some assumption (see the first inequality in \eqref{e:lipcase1a}from Proposition \ref{p:lipcase1}), if the Dirichlet energy of $u$ in a small ball is big enough, then $u$ keeps the same sign in a smaller ball, which in view of the preceding analysis implies that $u$ is Lipschitz continuous.

In subsection \ref{sub:endlipcont} we complete the proof of the Lipschitz continuity of $u$. The main missing step is to deal with the case where the Dirichlet energy is big and the first assumption of \eqref{e:lipcase1a} in Proposition \ref{p:lipcase1} fails. Using an almost-monotonicity formula for operators with variable coefficients proved by Matevosyan and Petrosyan in \cite[Theorem III]{matevosyan-petrosyan-11}, we show that in this case the value of the Dirichlet energy has to decrease at some smaller scale. 
\medskip

Throughout this section we fix $u:=u_i$, for some $i=1,\dots,k$, a coordinate function of the vector $U=(u_1,\dots,u_k)$ from Theorem \ref{t:main2}. We start by proving that $u$ is a bounded function in $D$.

\begin{lm}[Boundedness]\label{l:degiorgi}
Let $\O\subset D$ be a (non-empty) quasi-open set, $f\in L^p(D)$ for some $p\in (d/2,+\infty]$ and let $u\in H^1_0(\Omega)$ be the solution of 
\begin{equation*}
-\dive(A \nabla u)=f\quad\text{in}\quad\Omega,\qquad u\in H^1_0(\Omega). 
\end{equation*}
Then, there is a dimensional constant $C_d$ such that
$$\|u\|_{L^\infty}\le \frac{\lambda_{\text{\tiny\sc A}}^2 C_d }{\sfrac2d-\sfrac1p}|\Omega|^{\sfrac2d-\sfrac1p}\|f\|_{L^p}.$$
\end{lm}
\begin{proof}
Up to arguing with the positive and the negative parts of $f$, we can assume that $f$ is a non-negative function. By the maximum principle (see \cite[Theorem 8.1]{gilbarg-trudinger-01}) we have $u\ge 0$ on $\Omega$. Moreover, $u$ is a minimum of the following functional
$$J(\varphi):=\frac12\int_\Omega A\nabla\varphi\cdot\nabla\varphi-\int_\Omega f\varphi, \qquad \varphi \in H^1_0(\O).$$ 
We consider, for every $0<t<\|u\|_{L^\infty}$ and $\eps>0$, the test function $u_{t,\eps}=u\wedge t +(u-t-\eps)_+\in H^1_0(\O)$. Then, by ellipticity of the matrices $A_x$ and the inequality $J(u)\le J(u_{t,\eps})$ we get that 
\begin{align*}
\ds\frac{1}{2\lambda_{\text{\tiny\sc A}}^2}\int_{\{t< u\le t+\eps\}}|\nabla u|^2 & \le \ds\frac12\int_{\{t< u\le t+\eps\}}A\nabla u\cdot\nabla u
\le \int_{\R^d} f\left(u-u_{t,\eps}\right) \\
&\le\eps \int_{\{u>t\}}f\le \eps \|f\|_{L^p}|\{u>t\}|^{\frac{p-1}p},
\end{align*}
The end of the proof now follows precisely as in \cite[Lemma 5.3]{russ-trey-velichkov-19}.
\end{proof}

\subsection{Continuity and H\"{o}lder continuity}\label{sub:contholdcont}
We change the coordinates and reduce to the case $A=\text{Id}$ using in particular the H\"{o}lder continuity of the coefficients of $A$, and we then prove that $u$ is locally H\"{o}lder continuous in $D$.
Let us first introduce few notations that we will use throughout this section. For $x\in D$ we define the function $F_x : \R^d\rightarrow \R^d$ by
\begin{equation*}
F_x(\xi):=x+A_x^{\sfrac12}[\xi],\qquad \xi\in\R^d.
\end{equation*}
Moreover, we set $u_x=u\circ F_x$ for every $x\in D$.

\begin{oss}
For $M\in Sym_d^+$ we denote by $M^{\sfrac12}$ the square root matrix of $M$. We recall that if $M\in Sym_d^+$, then there is an orthogonal matrix $P$ such that $PMP^t=\text{diag}(\lambda_1,\dots,\lambda_d)$, where $P^t$ is the transpose of $P$ and $\text{diag}(\lambda_1,\dots,\lambda_d)$ is the diagonal matrix with eigenvalues $\lambda_1,\dots,\lambda_d$. The matrix $M^{\sfrac12}$ is then defined by $M^{\sfrac12}:=P^tDP$ where $D=\text{diag}(\sqrt\lambda_1,\dots,\sqrt\lambda_d)$.
\end{oss}

\begin{oss}[Notation of the harmonic extension]
One of the main ingredient in the proof of Theorem \ref{t:main2} is based on small variations of the function $u_x$. Precisely, we will often compare $u_x$ in some ball $B_r$ with the harmonic extension of the trace of $u_x$ to $\partial B_r$. This function will often be denoted by $h_{x,r}$, or more simply $h_r$ if there is no confusion, and is defined by $h_r=h_{x,r} \in H^1(B_r)$ and
\[ \Delta h_r=0\quad\text{in}\quad B_r,\qquad u_x-h_r\in H^1_0(B_r). \]
We notice that $h_r$ is a minimizer of the Dirichlet energy in the ball $B_r$, that is
\[ \int_{B_r}|\nabla h_r|^2 \leq \int_{B_r}|\nabla v|^2\qquad\text{for every } v\in H^1(B_r)\text{ such that } h_r-v \in H^1_0(B_r). \]
\end{oss}

We now prove that the function $u_x$ is in some sense a quasi-minimizer for the Dirichlet energy in small balls centred at the origin.

\begin{prop}\label{p:quminux}
There exist constants $r_0\in (0,1)$ and $C>0$ such that, if $x\in D$ and $r\leq r_0$ satisfy $B_{\lambda_{\text{\tiny\sc A}} r}(x) \subset D$, then we have
\begin{equation}\label{e:quminux0}
\int_{B_r}|\nabla u_x|^2 \leq (1+Cr^{\delta_{\text{\tiny\sc A}}})\int_{B_r}|\nabla \tilde{u}|^2 + Cr^d,
\end{equation}
for every $\tilde{u} \in H^1(\R^d)\cap L^\infty(\R^d)$ such that $u_x-\tilde{u} \in H^1_0(B_r)$ and $\|\tilde{u}\|_{L^\infty}\leq \|u\|_{L^\infty}$.
\end{prop}

\begin{proof}
Let $v\in H^1_0(D)$ be such that $\tilde{u}=v\circ F_x$ and set $\tilde{U}=(u_1,\dots,v,\dots,u_k)\in H^1_0(D,\R^k)$, where $v$ stands at the $i$-th position. Set $\rho=\lambda_{\text{\tiny\sc A}} r$ and note that $F_x(B_r)\subset B_\rho(x)\subset D$. Then, using $\tilde{U}$ as a test function and observing that $u-v \in H^1_0(F_x(B_r))$, we get
\[ \int_{F_x(B_r)}A\nabla u\cdot\nabla u \leq \int_{F_x(B_r)}A\nabla v\cdot\nabla v + C\|u-v\|_{L^1} \int_DA\nabla\tilde{U}\cdot\nabla\tilde{U} + \Lambda|B_{\rho}|, \]
where $C$ is the constant from Theorem \ref{t:main2}. Together with
\begin{align*}
\int_DA\nabla\tilde{U}\cdot\nabla\tilde{U} &\leq \int_DA\nabla U\cdot\nabla U - \int_DA\nabla u\cdot\nabla u + \int_DA\nabla v\cdot\nabla v \\
&\leq \int_DA\nabla U\cdot\nabla U + \int_{F_x(B_r)}A\nabla v\cdot\nabla v
\end{align*}
this yields
\begin{equation}\label{e:quminux1}
\int_{F_x(B_r)}A\nabla u\cdot\nabla u \leq (1+\tilde{C}r^d)\int_{F_x(B_r)}A\nabla v\cdot\nabla v + \tilde{C}r^d,
\end{equation}
for some constant $\tilde{C}$. On the other hand, using the H\"{o}lder continuity and the ellipticity of $A$ we estimate
\begin{align}\label{e:quminux2}
\nonumber\int_{B_r}|\nabla u_x|^2 &= \det(A_x^{-\sfrac12})\int_{F_x(B_r)}A_x\nabla u\cdot\nabla u \\
&\leq \det(A_x^{-\sfrac12})(1+dc_{\text{\tiny\sc A}}\lambda_{\text{\tiny\sc A}}^2\rho^{\delta_{\text{\tiny\sc A}}})\int_{F_x(B_r)}A\nabla u\cdot\nabla u.
\end{align}
Similarly, we have the following estimate from below
\begin{equation}\label{e:quminux3}
\int_{B_r}|\nabla\tilde{u}|^2 \geq \det(A_x^{-\sfrac12})(1-dc_{\text{\tiny\sc A}}\lambda_{\text{\tiny\sc A}}^2\rho^{\delta_{\text{\tiny\sc A}}})\int_{F_x(B_r)}A\nabla v\cdot\nabla v.
\end{equation}
Now, combining \eqref{e:quminux2}, \eqref{e:quminux1} and \eqref{e:quminux3} we get
\[ \int_{B_r}|\nabla u_x|^2  \leq (1+dc_{\text{\tiny\sc A}}\lambda_{\text{\tiny\sc A}}^2\rho^{\delta_{\text{\tiny\sc A}}})\bigg[ \frac{1+\tilde{C}r^d}{1-dc_{\text{\tiny\sc A}}\lambda_{\text{\tiny\sc A}}^2\rho^{\delta_{\text{\tiny\sc A}}}}\int_{B_r}|\nabla\tilde{u}|^2 + \lambda_{\text{\tiny\sc A}}^d\tilde{C}r^d \bigg]. \]
which gives \eqref{e:quminux1}.
\end{proof}

We now prove that the function $u$ is continuous in $D$. In the sequel we will often use the following notation: for $x\in D$ and $r>0$ we set
\begin{equation*}\label{e:defomega}
\omega(u,x,r) = \left(\aver{B_r(x)}|\nabla u|^2\right)^{1/2}\qquad\text{ and }\qquad\omega(u_x,r) = \left(\aver{B_r}|\nabla u_x|^2\right)^{1/2}.
\end{equation*}

\begin{prop}\label{p:contu}
The function $u$ is continuous in $D$. Moreover, there exist $r_0>0$ and $C>0$ such that, if  $x\in D$ and $r \leq r_0$ satisfy $B_{r}(x)\subset D$, then we have 
\begin{equation}\label{e:contu}
|u(y)-u(z)|\leq C\Big(1+\omega(u,x,r)+\log\frac{r}{|y-z|}\Big)|y-z|\qquad\text{for every}\qquad y,z\in B_{r/2}(x).
\end{equation}
\end{prop}

The next Lemma shows that $\omega(u_x,r)$ cannot grow too fast as $r$ tends to zero  and will be useful throughout the proof of the Lipschitz continuity of $u$.

\begin{lm}\label{l:estcontu}
There exist constants $r_0>0$ and $C>0$ such that, if $x\in D$ and $r\leq r_0$ satisfy $B_{\lambda_{\text{\tiny\sc A}}r}(x)\subset D$, then we have
\begin{equation}\label{e:estcontua}
\omega(u_x,s)\leq C\omega(u_x,r)+C\log\Big(\frac{r}{s}\Big)\qquad\text{for every}\qquad 0<s\leq r.
\end{equation}
If, moreover, $x$ is a Lebesgue point for $u$, then we have
\begin{equation}\label{e:estcontub}
\bigg|u(x)-\aver{B_r}u_x\bigg| \leq Cr(1+\omega(u_x,r)).
\end{equation}
\end{lm}

\begin{proof}
Let $t\leq r$ and use $h_t$ as a test function in \eqref{e:quminux0}, where $h_t=h_{x,t}$ denotes the harmonic extension in $B_t$ of the trace of $u_x$ to $\partial B_t$, to get
\begin{align}\label{e:estcontu1}
\nonumber\int_{B_t}|\nabla(u_x-h_t)|^2 &= \int_{B_t}|\nabla u_x|^2 - \int_{B_t}|\nabla h_t|^2 \\
&\leq Ct^{\delta_{\text{\tiny\sc A}}}\int_{B_t}|\nabla h_t|^2 + Ct^d \leq Ct^{\delta_{\text{\tiny\sc A}}}\int_{B_t}|\nabla u_x|^2 + Ct^d,
\end{align}
where in the last inequality we have used that $h_t$ is a minimizer of the Dirichlet energy on $B_t$. Moreover, since $|\nabla h_t|$ is subharmonic on $B_s$ for every $s\leq t$, we have 
\begin{equation}\label{e:estcontu2}
\aver{B_s}|\nabla h_t|^2 \leq \aver{B_t}|\nabla h_t|^2\qquad\text{for every}\qquad s\leq t.
\end{equation}
Therefore, the triangle inequality, \eqref{e:estcontu2} and \eqref{e:estcontu1} give for every $s\leq t\leq r_0$
\begin{align}\label{e:estcontu3}
\nonumber\omega(u_x,s) &\leq \left(\aver{B_s}|\nabla(u_x-h_t)|^2\right)^{1/2} + \left(\aver{B_s}|\nabla h_t|^2\right)^{1/2} \\
\nonumber&\leq\Big(\frac{t}{s}\Big)^{d/2}\left(\aver{B_t}|\nabla(u_x-h_t)|^2\right)^{1/2} + \left(\aver{B_t}|\nabla h_t|^2\right)^{1/2} \\
&\leq \Big(\frac{t}{s}\Big)^{d/2}C\Big(t^{\delta_{\text{\tiny\sc A}}/2}\omega(u_x,t)+1\Big) +\omega(u_x,t)\\
\nonumber&\leq \Big(1+C\Big(\frac{t}{s}\Big)^{d/2}t^{\delta_{\text{\tiny\sc A}}/2}\Big)\omega(u_x,t) + C\Big(\frac{t}{s}\Big)^{d/2}.
\end{align}
We then use the estimate \eqref{e:estcontu3} with the radii $r_i=2^{-i}r$, $i\geq 0$, and we get
\[ \omega(u_x,r_i) \leq \Big(1+Cr_{i-1}^{\delta_{\text{\tiny\sc A}}/2}\Big)\omega(u_x,r_{i-1})+C, \qquad i\geq 1. \]
This, with an iteration, implies that for every $i\geq 1$ we have
\begin{align}\label{e:estcontu4}
\nonumber\omega(u_x,r_i) &\leq \omega(u_x,r)\prod_{j=0}^{i-1}\Big(1+Cr_j^{\delta_{\text{\tiny\sc A}}/2}\Big)+C\sum_{j=1}^{i-1}\prod_{l=j}^{i-1}\Big(1+Cr_l^{\delta_{\text{\tiny\sc A}}/2}\Big)+C\\
&\leq C\omega(u_x,r)+Ci,
\end{align}
where we used that the product $\prod_{j=0}^{\infty}\big(1+Cr_j^{\delta_{\text{\tiny\sc A}}/2}\big)$ is bounded by a constant depending on $r_0$. 
The first estimate of the Lemma now follows from \eqref{e:estcontu4}. Indeed, choose $i\geq 0$ such that $r_{i+1}<s\leq r_i$ and note that we have $\omega(u_x,s)\leq 2^{d/2}\omega(u_x,r_i)$. If $i=0$, this directly implies \eqref{e:estcontua}; otherwise, $i\geq 1$ and use also \eqref{e:estcontu4}. 

We now prove the second estimate. For $i\geq 0$ we set $m_i=\aver{B_{r_i}}u_x$.
By the Poincar\'{e} inequality and \eqref{e:estcontu4} we have
\begin{equation}\label{e:estcontu5}
\left(\aver{B_{r_i}}|u_x-m_i|^2\right)^{1/2} \leq Cr_i\omega(u_x,r_i)\leq Cr_i(\omega(u_x,r)+i).
\end{equation}
Furthermore, $0$ is a Lebesgue point for $u_x$ since $x$ is a Lebesgue point for $u$ and that for every $s\leq r$ we have
\[ \lambda_{\text{\tiny\sc A}}^{-2d}\aver{B_{\lambda_{\scalebox{.5}{A}}^{-1}s}(x)}|u-u(x)| \leq \aver{B_s}|u_x-u_x(0)|=\aver{F_x(B_s)}|u-u(x)| \leq \lambda_{\text{\tiny\sc A}}^{2d}\aver{B_{\lambda_{\scalebox{.5}{A}}s}(x)}|u-u(x)|. \]
In particular, it follows that $m_i$ converges to $u_x(0)=u(x)$ as $i\rightarrow+\infty$. Therefore, this with the Cauchy-Schwarz inequality and \eqref{e:estcontu5} give
\begin{align*}
|u(x)-m_i| &\leq \sum_{j=i}^{+\infty}|m_{j+1}-m_j| \leq \sum_{j=i}^{+\infty}\aver{B_{r_{j+1}}}|u_x-m_j| \\
&\leq 2^d \sum_{j=i}^{+\infty}\aver{B_{r_j}}|u_x-m_j| \leq 2^d \sum_{j=i}^{+\infty}\left(\aver{B_{r_j}}|u_x-m_j|^2\right)^{1/2} \\
&\leq C\sum_{j=i}^{+\infty}r_j(\omega(u_x,r)+j) \leq Cr_i(\omega(u_x,r)+i+1),
\end{align*}
where in the last inequality we used that $\sum_{j=i}^{+\infty}2^{i-j}j\leq C(i+1)$. Then, observe that \eqref{e:estcontub} is precisely the above inequality with $i=0$ to conclude the proof.
\end{proof}

\begin{proof}[Proof of Proposition \ref{p:contu}]
Let $y,z\in B_{r/2}(x)$ and notice that it is enough to prove \eqref{e:contu} when $y$ and $z$ are Lebesgue points for $u$. Set $\delta=|y-z|$. We first assume that $4\lambda_{\text{\tiny\sc A}}^2\delta\leq r$.
Observe that we hence have the inclusions $F_z(B_{{\lambda_{\text{\scalebox{.8}{A}}}^{-1}}\delta}) \subset F_y(B_{2\lambda_{\text{\tiny\sc A}}\delta})\subset B_{r}(x)\subset D$. Using a change of variables, the Poincar\'{e} inequality and then the ellipticity of $A$, we estimate
\begin{align}\label{e:contu1}
\nonumber\bigg|\aver{B_{2\lambda_{\scalebox{.5}{A}}}\delta}u_y - \aver{B_{\lambda_{\scalebox{.5}{A}}^{-1}\delta}}u_z \bigg| 
&= \bigg|\aver{F_y(B_{2\lambda_{\scalebox{.5}{A}}\delta})}u - \aver{F_z(B_{\lambda_{\scalebox{.5}{A}}^{-1}\delta})}u \bigg| 
\leq  \aver{F_z(B_{\lambda_{\scalebox{.5}{A}}^{-1}\delta})}\Big|u-\aver{F_y(B_{2\lambda_{\text{\tiny\sc A}}\delta})}u\Big|  \\
&\leq 2^d\lambda_{\text{\tiny\sc A}}^{4d}  \aver{F_y(B_{2\lambda_{\scalebox{.5}{A}}\delta})}\Big|u-\aver{F_y(B_{2\lambda_{\scalebox{.5}{A}}\delta})}u\Big|  \leq C\delta\left(\aver{F_y(B_{2\lambda_{\scalebox{.5}{A}}\delta})}|\nabla u|^2\right)^{1/2} \\&
\nonumber\leq C\delta\lambda_{\text{\tiny\sc A}} \left(\aver{F_y(B_{2\lambda_{\scalebox{.5}{A}}\delta})}A_y\nabla u\cdot\nabla u\right)^{1/2} \leq C\delta\omega(u_y,2\lambda_{\text{\tiny\sc A}}\delta).
\end{align}
On the other hand, since $F_z(B_{\lambda_{\scalebox{.5}{A}}^{-1}\delta})\subset F_y(B_{2\lambda_{\scalebox{.5}{A}}\delta})$ we have
\begin{align}\label{e:contu2}
\nonumber\omega(u_z,\lambda_{\text{\tiny\sc A}}^{-1}\delta) &= \left(\aver{F_z(B_{\lambda_{\scalebox{.5}{A}}^{-1}\delta})}A_z\nabla u\cdot\nabla u\right)^{1/2} 
\leq \lambda_{\text{\tiny\sc A}}\left(\aver{F_z(B_{\lambda_{\scalebox{.5}{A}}^{-1}\delta})}|\nabla u|^2\right)^{1/2} \\
&\leq 2^{d/2}\lambda_{\text{\tiny\sc A}}^{2d+1}\left(\aver{F_y(B_{2\lambda_{\scalebox{.5}{A}}\delta})}|\nabla u|^2\right)^{1/2} 
\leq 2^{d/2}\lambda_{\text{\tiny\sc A}}^{2d+2}\left(\aver{F_y(B_{2\lambda_{\scalebox{.5}{A}}\delta})}A_y\nabla u\cdot\nabla u\right)^{1/2} \\ 
\nonumber&\leq C\omega(u_y,2\lambda_{\text{\tiny\sc A}}\delta).
\end{align}
We now apply \eqref{e:estcontub} to get
\begin{equation}\label{e:contu3}
\bigg|u(y)-\aver{B_{2\lambda_{\scalebox{.5}{A}}\delta}}u_y\bigg| \leq C\delta(\omega(u_y,2\lambda_{\text{\tiny\sc A}}\delta)+1)
\end{equation}
and
\begin{equation}\label{e:contu4}
\bigg|u(z)-\aver{B_{\lambda_{\scalebox{.5}{A}}^{-1}\delta}}u_z\bigg| \leq C\delta(\omega(u_z,\lambda_{\text{\tiny\sc A}}^{-1}\delta)+1) \leq C\delta(\omega(u_y,2\lambda_{\text{\tiny\sc A}}\delta)+1),
\end{equation}
where we used \eqref{e:contu2} in the last inequality. Therefore, combining the triangle inequality, \eqref{e:contu3}, \eqref{e:contu1} and \eqref{e:contu4} we get that
\begin{equation}\label{e:contu8}
|u(y)-u(z)| \leq C\delta(\omega(u_y,2\lambda_{\text{\tiny\sc A}}\delta)+1).
\end{equation}
Moreover, by \eqref{e:estcontua} (recall that we assumed $4\lambda_{\text{\tiny\sc A}}^2\delta\leq r$) we have
\begin{equation}\label{e:contu5}
\omega(u_y,2\lambda_{\text{\tiny\sc A}}\delta) \leq C\omega(u_y,(2\lambda_{\text{\tiny\sc A}})^{-1}r) + C\log\frac{r}{4\lambda_{\text{\tiny\sc A}}^2\delta}.
\end{equation}
By the ellipticity of $A$ and since $F_y(B_{{\scriptscriptstyle(2\lambda_{\scalebox{.5}{A}})^{-1}}r})\subset B_{r}(x)$, we have the following estimate
\begin{align}\label{e:contu6}
\nonumber\omega(u_y,(2\lambda_{\text{\tiny\sc A}})^{-1}r) &=\left(\aver{F_y(B_{{\scriptscriptstyle(2\lambda_{\scalebox{.5}{A}})^{-1}}r})}A_y\nabla u\cdot\nabla u\right)^{1/2} \leq \lambda_{\text{\tiny\sc A}}\left(\aver{F_y(B_{(2\lambda_{\scalebox{.5}{A}})^{-1}r})}|\nabla u|^2\right)^{1/2} \\
&\leq 2^{d/2}\lambda_{\text{\tiny\sc A}}^{d+1}\left(\aver{B_{r}(x)}|\nabla u|^2\right)^{1/2} \leq C\omega(u,x,r).
\end{align} 
Finally, combine \eqref{e:contu8}, \eqref{e:contu5} and \eqref{e:contu6} to get
\begin{align}\label{e:contu7}
|u(y)-u(z)| &\leq C\delta\Big(1+\omega(u,x,r)+ \log\frac{r}{4\lambda_{\text{\tiny\sc A}}^2\delta}\Big) \\
\nonumber&\leq C|y-z|\Big(1+\omega(u,x,r)+\log\frac{r}{|y-z|}\Big),
\end{align}
which is \eqref{e:contu}.

Now, if the assumption $4\lambda_{\text{\tiny\sc A}}^2|y-z|\leq r$ is not satisfied, choose $n$ points $y_1=y,y_2,\dots,y_n=z$ in $B_{r}(x)$ such that $4\lambda_{\text{\tiny\sc A}}^2\eta=|y-z|$, where we have set $\eta=|y_i-y_{i+1}|$, $i=1,\dots,n$. Then we have $4\lambda_{\text{\tiny\sc A}}^2\eta\leq r$. We notice that we can assume the $y_i$ to be Lebesgue points for $u$. Moreover, observe that we can bound the number of points by $n\leq 16\lambda_{\text{\tiny\sc A}}^4+2$. Therefore, applying the estimate \eqref{e:contu7} to each pair $(y_i,y_{i+1})$ we have
\begin{align*}
|u(y)-u(z)| &\leq \sum_{i=1}^{n-1} |u(y_i)-u(y_{i+1})|
\leq C\sum_{i=1}^{n-1}\eta\Big(1+\omega(u,x,r)+ \log\frac{r}{4\lambda_{\text{\tiny\sc A}}^2\eta}\Big) \\
&\leq nC\frac{|y-z|}{4\lambda_{\text{\tiny\sc A}}^2}\Big(1+\omega(u,x,r)+ \log\frac{r}{|y-z|}\Big),
\end{align*}
which concludes the proof.
\end{proof}

We are now in position to prove the H\"{o}lder continuity of $u$. 

\begin{prop}\label{p:holcontu}
The function $u$ is locally $\alpha$-H\"{o}lder continuous in $D$ for every $\alpha\in (0,1)$, that is, for every compact set $K\subset D$, there exist $r_K>0$ and $C_K>0$ such that for every $x\in K$ we have
\begin{equation}\label{e:holcontu}
|u(y)-u(z)|\leq C_K|y-z|^\alpha\qquad\text{for every}\qquad y,z\in B_{r_K}(x).
\end{equation}
\end{prop}

\begin{proof}
Let $x\in K$ and set $4r_K=r_1 = \min\{r_0,\text{dist}(K,D^c)\}$ where $r_0$ is given by Proposition \ref{p:contu}. Since the function $r\mapsto r^{1-\alpha}\log(r_1/r)$ is non-decreasing on $(0,c_\alpha)$ for some constant $c_\alpha>0$ depending on $\alpha$ and $r_1$,
it follows from Proposition \ref{p:contu} that, if $y,z \in B_{r_1/2}(x)$ are such that $|y-z|\leq c_\alpha$, we have
\begin{align}\label{e:holcontu1}
\nonumber|u(y)-u(z)| &\leq C\Big(r_1^{1-\alpha}(1+\omega(u,x,r_1)) + c_\alpha^{1-\alpha}\log\frac{r_1}{c_\alpha}\Big)|y-z|^\alpha \\
&\leq C(1+\omega(u,x,r_1))|y-z|^\alpha 
\end{align}
If now $|y-z|>c_\alpha$, then choose $n$ points $y_1=y,\dots,y_n=z$ in $B_{r_1/2}(x)$ such that $|y_i-y_{i+1}|=c_\alpha r_1^{-1}|y-z|$, with $n$ bounded by some constant depending on $\alpha$ and $r_1$. 
Then apply \eqref{e:holcontu1} to each pair $(y_i,y_{i+1})$ to prove that $u$ is $\alpha$-H\"{o}lder continuous in the ball $B_{r_1/2}(x)$ with a modulus of continuity depending on $\omega(u,x,r_1)$. 
Now, \eqref{e:holcontu} follows by a compactness argument with the constant $C_K$ depending on $\max\{\omega(u,x_i,r_1),\, i=\dots N\}$, where the $x_i$'s are given by some subcovering of $K\subset \cup_{i=1}^N B_{r_K}(x_i)$. 
\end{proof}

\subsection{Bound of the Lipschitz constant in $\{u>0\}$}\label{sub:boundlipcstustrpos}
We prove that $u$ is Lipschitz continuous and even $C^{1,\beta}$-regular in the regions where $u$ keeps the same sign. We also provide in this case an estimate of the Lipschitz constant of $u$ in terms of $\omega(u,x,r)$ (see Proposition \ref{p:strposlip}. Then, we show that under suitable conditions, $u$ keeps the same sign and is therefore Lipschitz continuous (see Proposition \ref{p:lipcase1}).

\begin{prop}\label{p:strposlip}
Let $K\subset D$ be a compact set. There exist constants $r_K>0$ and $C_K>0$ such that, if $x\in K$ and $r\leq r_K$ satisfy
\begin{equation}\label{e:strposlipa}
\text{either}\qquad u_x>0 \text{ a.e. in } B_r\qquad\text{or}\qquad u_x<0 \text{ a.e. in } B_r,
\end{equation}
then $u$ is Lipschitz continuous in $B_{r/2}(x)$ and we have
\begin{equation}\label{e:strposlipb}
|u(y)-u(z)|\leq C_K(1+\omega(u,x,r))|y-z|\qquad\text{for every}\qquad y,z\in B_{r/2}(x).
\end{equation}
Moreover, $u$ is $C^{1,\beta}$ in the ball $B_{r/4}(x)$ where $\beta=\frac{\delta_{\scalebox{.5}{A}}}{d+\delta_{\scalebox{.5}{A}}+2}$ and we have
\begin{equation}\label{e:strposlipc}
|\nabla u(y)-\nabla u(z)|\leq C_Kr^{-\frac{\delta_{\scalebox{.5}{A}}}{d+2}}(1+\omega(u,x,r))|y-z|^\beta\qquad\text{for every}\qquad y,z\in B_{r/4}(x).
\end{equation}
\end{prop}

In the next Lemma we compare the Dirichlet energy of $u_x$ and of its harmonic extension in small balls where $u_x$ has a given sign. The estimate \eqref{e:estuxhxrb} in Lemma \ref{l:estuxhxr} below is similar to \eqref{e:quminux0} but with a smaller error term. Thanks to this improvement, the strategy developed in the proof of Lemma \ref{l:estcontu} will lead to a sharper result than estimate \eqref{e:estcontua}, namely \eqref{e:strposlipb}. 

\begin{lm}\label{l:estuxhxr}
Let $K\subset D$ be a compact set and let $\alpha\in (0,1)$. There exist constants $r_K>0$ and $C>0$ such that, if $x\in K$ and $r\leq r_K$ are such that \eqref{e:strposlipa} holds, then the function $u_x=u\circ F_x$ satisfies
\begin{equation}\label{e:estuxhxrb}
\int_{B_r}|\nabla u_x|^2\leq(1+Cr^{\delta_{\text{\tiny\sc A}}})\int_{B_r}|\nabla h_r|^2 + Cr^{d+\alpha},
\end{equation}
where $h_r$ stands for the harmonic extension of the trace of  $u_x$ to $\partial B_r$.
\end{lm}

\begin{proof}
Set $\rho:=\lambda_{\text{\tiny\sc A}}r$ for some $r>0$ small enough so that $B_{\rho}(x)\subset D$. We define $v\in H^1_0(D)$ by $h_r=v\circ F_x$ in $B_r$ and $v=u$ elsewhere so that we have $u-v\in H^1_0(F_x(B_r))$. Set $\tilde{U}=(u_1,\dots,v,\dots,u_k)\in H^1_0(D,\R^k)$ and observe that $|\{|\tilde{U}|>0\}|=|\{|U|>0\}|$ by \eqref{e:strposlipa} and because $v>0$ in $F_x(B_r)$. Then, we use $\tilde{U}$ as a test function in \eqref{e:main2a} to get 
\[ \int_{F_x(B_r)}A\nabla u\cdot\nabla u \leq \int_{F_x(B_r)}A\nabla v\cdot\nabla v + C\|u-v\|_{L^1} \int_DA\nabla\tilde{U}\cdot\nabla\tilde{U}, \]
where $C$ is the constant from Theorem \ref{t:main2}.
Now, since $u$ is locally $\alpha$-H\"{o}lder continuous, we have the bound $\|u-v\|_{L^1}\leq C_dC_Kr^{d+\alpha}$, where the constant $C_K$ is given by Proposition \ref{p:holcontu}. Moreover we have the estimate
\[ \int_DA\nabla\tilde{U}\cdot\nabla\tilde{U} \leq \int_DA\nabla U\cdot\nabla U + \int_{F_x(B_r)}A\nabla v\cdot\nabla v. \]
Altogether this gives
\begin{equation*}
\int_{F_x(B_r)}A\nabla u\cdot\nabla u \leq (1+\tilde{C}r^{d+\alpha})\int_{F_x(B_r)}A\nabla v\cdot\nabla v + \tilde{C}r^{d+\alpha},
\end{equation*}
for some constant $\tilde{C}$ which involves $\int_DA\nabla U\cdot\nabla U$. Finally, using the H\"{o}lder continuity and the ellipticity of $A$ as in the proof of Proposition \ref{p:quminux}, we get
\[ \int_{B_r}|\nabla u_x|^2  \leq (1+dc_{\text{\tiny\sc A}}\lambda_{\text{\tiny\sc A}}^2\rho^{\delta_{\text{\tiny\sc A}}})\bigg[ \frac{1+\tilde{C}r^{d+\alpha}}{1-dc_{\text{\tiny\sc A}}\lambda_{\text{\tiny\sc A}}^2\rho^{\delta_{\text{\tiny\sc A}}}}\int_{B_r}|\nabla h_r|^2 + \lambda_{\text{\tiny\sc A}}^d\tilde{C}r^{d+\alpha} \bigg]. \]
which gives \eqref{e:estuxhxrb}.
\end{proof}

Next Lemma is analogue to Lemma \ref{l:estcontu} with a better estimate of the error term. Its proof is quite similar but we nonetheless sketch the argument since there are small differences. 

\begin{lm}\label{l:estholcontu}
Let $K\subset D$ be a compact set and $\alpha\in (0,1)$. There exist constants $r_K>0$ and $C>0$ such that, for every $x\in K$ and every $r\leq r_K$ such that \eqref{e:strposlipa} holds, we have
\begin{equation}\label{e:estholcontua}
\omega(u_x,s)\leq C\omega(u_x,r)+Cr^{\alpha/2}\qquad\text{for every}\qquad 0<s\leq r.
\end{equation}
If, moreover, $x$ is a Lebesgue point for $u$, we have
\begin{equation}\label{e:estholcontub}
\bigg|u(x)-\aver{B_r}u_x\bigg| \leq Cr(\omega(u_x,r)+r^{\alpha/2}).
\end{equation}
\end{lm}

\begin{proof}
For $t\leq r\leq r_K$ we have by Lemma \ref{l:estuxhxr}
\begin{equation}\label{e:estholcontu1}
\int_{B_t}|\nabla(u_x-h_t)|^2 = \int_{B_t}|\nabla u_x|^2 - \int_{B_t}|\nabla h_t|^2 \leq Ct^{\delta_{\text{\tiny\sc A}}}\int_{B_t}|\nabla u_x|^2 + Ct^{d+\alpha},
\end{equation}
since $h_t$ is a minimizer of the Dirichlet energy on $B_t$. 
Now, for $s\leq t\leq r_0$ we use \eqref{e:estcontu2} and \eqref{e:estholcontu1} to estimate as in \eqref{e:estcontu3}
\begin{equation*}
\nonumber\omega(u_x,s)\leq \Big(1+C\Big(\frac{t}{s}\Big)^{d/2}t^{\delta_{\text{\tiny\sc A}}/2}\Big)\omega(u_x,t) + C\Big(\frac{t}{s}\Big)^{d/2}t^{\alpha/2},
\end{equation*}
which, applied to $s=2^{-i}r$ and $t=2^{-(i-1)}r$, gives
\[ \omega(u_x,r_i) \leq \Big(1+Cr_{i-1}^{\delta_{\text{\tiny\sc A}}/2}\Big)\omega(u_x,r_{i-1})+Cr_{i-1}^{\alpha/2}, \qquad i\geq 1, \]
where we have set $r_i=2^{-i}r$.
Iterating the above estimate we get for every $i\geq 1$
\begin{align*}
\nonumber\omega(u_x,r_i) &\leq \omega(u_x,r)\prod_{j=0}^{i-1}\Big(1+Cr_j^{\delta_{\text{\tiny\sc A}}/2}\Big)+C\sum_{j=1}^{i-1}\bigg(r_{j-1}^{\alpha/2}\prod_{l=j}^{i-1}\Big(1+Cr_l^{\delta_{\text{\tiny\sc A}}/2}\Big)\bigg)+Cr_{i-1}^{\alpha/2}\\
&\leq C\omega(u_x,r)+Cr^{\alpha/2},
\end{align*}
since $\prod_{j=0}^{\infty}\big(1+Cr_j^{\delta_{\text{\tiny\sc A}}/2}\big)$ is bounded by a constant depending on $r_K$. This proves \eqref{e:estholcontua}.

Finally, \eqref{e:estholcontub} is proved in the same way than \eqref{e:estcontub} but with \eqref{e:estcontu5} replaced by the estimate
\begin{equation*}
\left(\aver{B_{r_i}}|u_x-m_i|^2\right)^{1/2} \leq Cr_i\omega(u_x,r_i)\leq Cr_i(\omega(u_x,r)+r^{\alpha/2}).
\end{equation*}
\end{proof}

\begin{proof}[Proof of Proposition \ref{p:strposlip}]
Let us first prove \eqref{e:strposlipb}. We follow the proof of Proposition \ref{p:contu} and we only detail the few differences.
Let $y,z\in B_{r/2}(x)$ be Lebesgue points for $u$ and set $\delta=|y-z|$. We first assume that $4\lambda_{\text{\tiny\sc A}}^2\delta \leq r$. By \eqref{e:estholcontub} we have
\begin{equation}\label{e:strposli1}
\bigg|u(y)-\aver{B_{2\lambda_{\scalebox{.5}{A}}\delta}}u_y\bigg| \leq C\delta(\omega(u_y,2\lambda_{\text{\tiny\sc A}}\delta)+\delta^{\alpha/2}),
\end{equation}
and, using also \eqref{e:contu2},
\begin{equation}\label{e:strposli2}
\bigg|u(z)-\aver{B_{\lambda_{\scalebox{.5}{A}}^{-1}\delta}}u_z\bigg| \leq C\delta(\omega(u_z,\lambda_{\text{\tiny\sc A}}^{-1}\delta)+r^{\alpha/2}) \leq C\delta(\omega(u_y,2\lambda_{\text{\tiny\sc A}}\delta)+\delta^{\alpha/2}).
\end{equation}
Moreover, by \eqref{e:estholcontua} we have
\begin{equation}\label{e:strposli3}
\omega(u_y,2\lambda_{\text{\tiny\sc A}}\delta) \leq C\omega(u_y,(2\lambda_{\text{\tiny\sc A}})^{-1}r) + Cr^{\alpha/2}.
\end{equation}
Then, combining \eqref{e:strposli1}, \eqref{e:contu1}, \eqref{e:strposli2} and then \eqref{e:strposli3} and \eqref{e:contu6} we have
\begin{align}\label{e:strposli4}
\nonumber|u(y)-u(z)| &\leq C\delta(\omega(u_y,2\lambda_{\text{\tiny\sc A}}\delta)+\delta^{\alpha/2}) 
\leq C\delta(1+\omega(u,x,r)+ r^{\alpha/2}) \\
&\leq C(1+\omega(u,x,r))|y-z|.
\end{align}
Finally, if $4\lambda_{\text{\tiny\sc A}}^2\delta > r$, we argue as in the proof of Proposition \ref{p:contu} and choose a few number of points which connect $y$ and $z$ to prove \eqref{e:strposli4}.

We now prove the estimate \eqref{e:strposlipc}. Let $y\in B_{r/4}(x)$ and $\bar{r}\leq\lambda_{\text{\tiny\sc A}}^{-1}r/4$. We set $m(u_y,\rho)=\aver{B_{\rho}}\nabla u_y$ for $\rho\leq\bar{r}$ and $m=\aver{B_{\bar{r}}}\nabla h_{y,\bar{r}}=\nabla h_{y,\bar{r}}(0)$, where $h_{y,\bar{r}}$ denotes the harmonic extension of the trace of $u_y$ to $\partial B_{\bar{r}}$. Let $\eta\in(0,1/4)$. We want to estimate
\begin{equation}\label{e:strposli5}
\aver{B_{\eta \bar{r}}}|\nabla u_y - m(u_y,\eta\bar{r})|^2 \leq \aver{B_{\eta \bar{r}}}|\nabla u_y - m|^2 \leq 2\,\aver{B_{\eta \bar{r}}}|\nabla(u_y - h_{y,\bar{r}})|^2 + 2\,\aver{B_{\eta \bar{r}}}|\nabla h_{y,\bar{r}}-m|^2.
\end{equation}
Firstly, by \eqref{e:estholcontu1} we have
\begin{align}\label{e:strposli6}
\nonumber\aver{B_{\eta \bar{r}}}|\nabla(u_y - h_{y,\bar{r}})|^2 \leq C(\eta\bar{r})^{-d}\aver{B_{\bar{r}}}|\nabla(u_y-h_{y,\bar{r}})|^2 
&\leq C(\eta\bar{r})^{-d}\Big(\bar{r}^{\delta_{\text{\tiny\sc A}}}\int_{B_{\bar{r}}}|\nabla u_y|^2 + \bar{r}^{d+\alpha} \Big) \\
&\leq C\eta^{-d}\bar{r}^{\delta_{\text{\tiny\sc A}}}\omega(u_y,\bar{r})^2 + C\eta^{-d}\bar{r}^\alpha.
\end{align}
Moreover, \eqref{e:strposli4} says that for almost every $z\in B_{r/4}(y)\subset B_{r/2}(x)$ we have $|\nabla u(z)|\leq C(1+\omega(u,x,r))$, which implies that
\begin{equation}\label{e:strposli7}
\omega(u_y,\bar{r})^2 = \aver{B_{\bar{r}}}|\nabla u_y|^2 \leq \lambda_{\text{\tiny\sc A}}^{2(d+1)}\aver{B_{r/4}(y)}|\nabla u|^2\leq C(1+\omega(u,x,r))^2.
\end{equation}
On the other hand, by estimates on harmonic functions (see \cite[Theorem 3.9]{gilbarg-trudinger-01}), the Cauchy-Schwarz inequality and \eqref{e:strposli7} we have for every $\xi\in B_{\eta\bar{r}}$
\begin{align}\label{e:strposli8}
\nonumber|\nabla h_{y,\bar{r}}(\xi)-m| &=|\nabla h_{y,\bar{r}}(\xi)-\nabla h_{y,\bar{r}}(0)| \leq \eta\bar{r}\sup_{B_{\eta\bar{r}}}|\nabla^2 h_{y,\bar{r}}| \leq C\eta\sup_{B_{2\eta\bar{r}}}|\nabla h_{y,\bar{r}}| \\
&\leq C\eta\bigg(\aver{B_{\bar{r}}}|\nabla h_{y,\bar{r}}|\bigg)\leq C\eta\bigg(\aver{B_{\bar{r}}}|\nabla h_{y,\bar{r}}|^2\bigg)^{1/2} \leq C\eta\bigg(\aver{B_{\bar{r}}}|\nabla u_y|^2\bigg)^{1/2}\\
\nonumber& \leq C\eta\,\omega(u_y,\bar{r}) \leq C\eta(1+\omega(u,x,r)),
\end{align}
where $\nabla^2h_{y,\bar{r}}$ stands for the Hessian matrix of $h_{y,\bar{r}}$. Therefore, combining \eqref{e:strposli5}, \eqref{e:strposli6}, \eqref{e:strposli7} and \eqref{e:strposli8} we get
\begin{align}\label{e:strposli9}
\nonumber\aver{B_{\eta \bar{r}}}|\nabla u_y - m(u_y,\eta\bar{r})|^2 &\leq C\eta^{-d}\bar{r}^{\delta_{\text{\tiny\sc A}}}(1+\omega(u,x,r))^2+C\eta^{-d}\bar{r}^\alpha + C\eta^2(1+\omega(u,x,r))^2 \\&\leq C(1+\omega(u,x,r))^2\Big[\eta^{-d}\bar{r}^{\delta_{\text{\tiny\sc A}}}+\eta^{-d}\bar{r}^\alpha+\eta^2\Big].
\end{align}
We set $\alpha=\delta_{\text{\tiny\sc A}}$ (recall that $\alpha\in(0,1)$ was arbitrary). Moreover, we set $\beta=\frac{\delta_{\text{\tiny\sc A}}}{d+\delta_{\text{\tiny\sc A}}+2}$ and $\eta=\bar{r}^{\frac{\delta_{\text{\tiny\sc A}}}{d+2}}$ so that we have $\eta^{-d}\bar{r}^{\delta_{\text{\tiny\sc A}}}=\eta^2=(\eta\bar{r})^{2\beta}$. Notice also that $\eta\bar{r}=\bar{r}^{1+\eps}$, where $\eps=\frac{\delta_{\text{\tiny\sc A}}}{d+2}$. Therefore, \eqref{e:strposli9} implies that for every $y\in B_{r/4}(x)$ and every $\rho\leq\big(\frac{r}{4\lambda_{\text{\tiny\sc A}}}\big)^{1+\eps}$ we have
\begin{equation}\label{e:strposli10}
\aver{B_{\rho}}|\nabla u_y - m(u_y,\rho)|^2 \leq C(1+\omega(u,x,r))^2\rho^{2\beta}.
\end{equation}
Now, let $y,z\in B_{r/4}(x)$ be Lebesgue points for $u$ and set $\delta=|x-y|$. 
We first assume that $2\lambda_{\text{\tiny\sc A}}^2\delta\leq\big(\frac{r}{4\lambda_{\text{\tiny\sc A}}}\big)^{1+\eps}$. Setting $\delta_i=2^{-i}\delta$, $i\geq 0$, we have using \eqref{e:strposli10}
\begin{align}\label{e:strposli11}
\nonumber|\nabla u_y(0)-m(u_y,\delta)|&\leq\sum_{i=0}^{+\infty}|m(u_y,\delta_{i+1})-m(u_y,\delta_i)|\leq\sum_{i=0}^{+\infty}\aver{B_{\delta_{i+1}}}|\nabla u_y-m(u_y,\delta_i)| \\
\nonumber&\leq 2^d\sum_{i=0}^{+\infty}\aver{B_{\delta_i}}|\nabla u_y-m(u_y,\delta_i)| \leq 2^d\sum_{i=0}^{+\infty}\bigg(\aver{B_{\delta_i}}|\nabla u_y-m(u_y,\delta_i)|^2\bigg)^{1/2} \\
&\leq C(1+\omega(u,x,r))\delta^\beta.
\end{align}
Similarly, we have
\begin{equation}\label{e:strposli12}
|\nabla u_z(0)-m(u_z,2\lambda_{\text{\tiny\sc A}}^2\delta)|\leq C(1+\omega(u,x,r))\delta^\beta.
\end{equation}
Moreover, using that $F_z^{-1}\circ F_y(B_\delta)\subset B_{2\lambda_{\scalebox{.5}{A}}^2\delta}$ we have
\begin{align}\label{e:strposli13}
\nonumber|m(u_y,\delta)-m(u_z,2\lambda_{\text{\tiny\sc A}}^2\delta)|&\leq \aver{B_\delta}|\nabla u_y-m(u_z,2\lambda_{\text{\tiny\sc A}}^2\delta)| \\
&\leq \aver{F_z^{-1}\circ F_y(B_\delta)}|A^{\sfrac12}_y A^{-\sfrac12}_z\nabla u_z-m(u_z,2\lambda_{\text{\tiny\sc A}}^2\delta)| \\
\nonumber&\leq (2\lambda_{\text{\tiny\sc A}}^2)^{2d}\aver{B_{2\lambda_{\scalebox{.5}{A}}\delta}}|A^{\sfrac12}_y A^{-\sfrac12}_z\nabla u_z-m(u_z,2\lambda_{\text{\tiny\sc A}}^2\delta)|.
\end{align}
Notice that the matrices $A^{\sfrac12}$ have H\"{o}lder continuous coefficients with exponent $\delta_{\text{\tiny\sc A}}/2$ and hence that $|A^{\sfrac12}_y A^{-\sfrac12}_z-\text{Id}|\leq\lambda_{\text{\tiny\sc A}}|A^{-\sfrac12}_y -A^{-\sfrac12}_z|\leq C\delta^{\delta_{\text{\tiny\sc A}}/2}\leq C\delta^\beta$ (because $\beta\leq\delta_{\text{\tiny\sc A}}/2$). Therefore, using \eqref{e:estholcontua} it follows that
\begin{align}\label{e:strposli14}
\nonumber\aver{B_{2\lambda_{\scalebox{.5}{A}}^2}\delta}|A^{\sfrac12}_y A^{-\sfrac12}_z\nabla u_z -\nabla u_z| &\leq C\delta^\beta\aver{B_{2\lambda_{\scalebox{.5}{A}}^2}\delta}|\nabla u_z| \leq C\delta^\beta\omega(u_z,2\lambda_{\text{\tiny\sc A}}^2\delta) \\
&\leq C\delta^\beta(\omega(u_x,\lambda_{\text{\tiny\sc A}}^{-1}r)+r^{\delta_{\text{\tiny\sc A}}}) \leq C(1+\omega(u,x,r))\delta^\beta.
\end{align}
Furthermore, the triangle inequality in \eqref{e:strposli13} together with \eqref{e:strposli14}, \eqref{e:strposli10} and Cauchy-Schwarz's inequality give
\begin{equation}\label{e:strposli15}
|m(u_y,\delta)-m(u_z,2\lambda_{\text{\tiny\sc A}}^2\delta)|\leq C(1+\omega(u,x,r))\delta^\beta.
\end{equation}
Now, \eqref{e:strposli11}, \eqref{e:strposli15} and \eqref{e:strposli12} infer
\begin{align*}
|\nabla u_y(0)-\nabla u_z(0)| &\leq |\nabla u_y(0)-m(u_y,\delta)| + |m(u_y,\delta)-m(u_z,2\lambda_{\text{\tiny\sc A}}^2\delta)| + |\nabla u_z(0)-m(u_z,2\lambda_{\text{\tiny\sc A}}^2\delta)| \\
&\leq C(1+\omega(u,x,r))\delta^\beta.
\end{align*}
Since $|\nabla u_y(0)|\leq\lambda_{\text{\tiny\sc A}}|\nabla u(y)|\leq C(1+\omega(u,x,r))$ for almost every $y\in B_{r/4}(x)$ by \eqref{e:strposli4}, we get
\begin{align*}
|\nabla u(y)-\nabla u(z)| &= |A^{-\sfrac12}_y\nabla u_y(0)-A^{-\sfrac12}_z\nabla u_z(0)| \\
&\leq |A^{-\sfrac12}_y\nabla u_y(0)-A^{-\sfrac12}_z\nabla u_y(0)| + |A^{-\sfrac12}_z\nabla u_y(0)-A^{-\sfrac12}_z\nabla u_z(0)| \\
&\leq |A^{-\sfrac12}_y-A^{-\sfrac12}_z||\nabla u_y(0)|+|A^{-\sfrac12}_z||\nabla u_y(0)-\nabla u_z(0)| \\
&\leq  C(1+\omega(u,x,r))\delta^\beta.
\end{align*}
If $|y-z|\geq \big(\frac{r}{4\lambda_{\text{\tiny\sc A}}}\big)^{1+\eps}$, then we can connect $y$ and $z$ through less than $\lambda_{\text{\tiny\sc A}}\big(\frac{4\lambda_{\text{\tiny\sc A}}}{r}\big)^{\eps}+2$ points. This shows \eqref{e:strposlipc} and concludes the proof.
\end{proof}

The strategy to prove Theorem \ref{t:main2} is to show that $\omega(u_x,r)$ cannot become too big as $r$ gets small. In the next Proposition we prove, under some condition (see the first inequality in \eqref{e:lipcase1a} below), that if $\omega(u_x,r)$ is big enough then $u$ keeps the same sign near the point $x$ and is hence Lipschitz continuous by Proposition \ref{p:strposlip}. The case where $\omega(u_x,r)$ is big and this condition fails is treated in the next subsection. We set for $x\in D$ and $r>0$ 
\begin{equation*}
b(u_x,r)=\aver{\partial B_r}u_x\,d\mathcal{H}^{d-1} \qquad\text{and}\qquad b^+(u_x,r)=\aver{\partial B_r}|u_x|\,d\mathcal{H}^{d-1}.
\end{equation*}

\begin{prop}\label{p:lipcase1}
Let $K\subset D$ be a compact set and let $\gamma>0$. There exists constants $r_K, C_K>0$ and $\kappa_1>0$ such that, if $x\in K$ and $r\leq r_K$ satisfy 
\begin{equation}\label{e:lipcase1a}
\gamma r(1+\omega(u_x,r)) \leq |b(u_x,r)|\qquad\text{and}\qquad \kappa_1\leq\omega(u_x,r),
\end{equation}
then there exists a constant $c>0$ (independent from $x$ and $r$) such that $u$ is Lipschitz continuous in $B_{cr/2}(x)$ and we have
\begin{equation}\label{e:lipcase1b}
|u(y)-u(z)|\leq C_K(1+\omega(u,x,r))|y-z|\qquad\text{for every}\qquad y,z\in B_{cr/2}(x).
\end{equation}
Moreover, $u$ is $C^{1,\beta}$ in $B_{cr/4}(x)$ where $\beta=\frac{\delta_{\text{\tiny\sc A}}}{d+\delta_{\text{\tiny\sc A}}+2}$ and
 we have
\begin{equation}\label{e:lipcase1c}
|\nabla u(y)-\nabla u(z)|\leq C_Kr^{-\frac{\delta_{\text{\tiny\sc A}}}{d+2}}(1+\omega(u,x,r))|y-z|^\beta\qquad\text{for every}\qquad y,z\in B_{cr/4}(x).
\end{equation}
\end{prop}

Roughly speaking, the condition \eqref{e:lipcase1a} says that the absolute value of the trace of $u_x$ to $\partial B_r$ is big. This will in fact ensure that $u_x$ has, in some smaller ball, the same sign than (the average of) $u_x$ on $\partial B_r$. 

\begin{lm}\label{l:firstlmlip}
Let $\gamma$ and $\tau$ be two positive constants. There exist $r_0, \eta \in (0,1)$ and $\kappa_1>0$ such that, if $x\in D$ and $r\leq r_0$ satisfy $B_{\lambda_{\text{\tiny\sc A}}r}(x)\subset D$,
\begin{equation}\label{e:firstlmlipa}
\gamma(1+\omega(u_x,r)) \leq \frac{1}{r}|b(u_x,r)|\qquad\text{and}\qquad \kappa_1\leq\omega(u_x,r),
\end{equation}
then there exist $\rho\in (\frac{\eta r}{2},\eta r)$ such that
\begin{equation}\label{e:firstlmlipb}
\tau(1+\rho^{\delta_{\text{\tiny\sc A}}/2}\omega(u_x,\rho))\leq \frac{1}{\rho}|b(u_x,\rho)|\qquad\text{and}\qquad b^+(u_x,\rho)\leq3|b(u_x,\rho)|.
\end{equation}
Moreover, $b(u_x,r)$ and $b(u_x,\rho)$ have the same sign.
\end{lm}

\begin{proof}
We first prove the second inequality in \eqref{e:firstlmlipb}. Let us recall that $h_r=h_{x,r}$ denotes the harmonic extension of the trace of $u_x$ to $\partial B_r$. We want to estimate both $\aver{\partial B_{\rho}}|h_r|$ and $\aver{\partial B_{\rho}}|u_x-h_r|$ in terms of $|b(u_x,r)|$ for some $\rho\in (\frac{\eta r}{2},\eta r)$ defined soon (by \eqref{e:firstlmlip6}).
If $\eta\leq 1/2$, then by subharmonicity of $|\nabla h_r|$ in $B_r$ we have that for every $\xi \in B_{\eta r}$
\[ |\nabla h_r(\xi)|^2 \leq \aver{B_{r/2}(\xi)}|\nabla h_r|^2 \leq 2^d\aver{B_r}|\nabla h_r|^2 \leq 2^d\omega(u_x,r)^2. \]
Moreover, $b(u_x,r)=b(h_r,r)=h_r(0)$ since $h_r$ is harmonic and hence, choosing $\eta$ such that $\eta 2^{d/2}\leq \gamma/4$, we get
\begin{align}\label{e:firstlmlip1}
\nonumber|b(u_x,r)-h_r(\xi)|&=|h_r(0)-h_r(\xi)|\leq \eta r\|\nabla h_r\|_{L^\infty(B_{\eta r})} \leq \eta r2^{d/2}\omega(u_x,r) \\
&\leq \frac{\gamma r}{4}\omega(u_x,r) \leq \frac{1}{4}|b(u_x,r)|,
\end{align}
where in the last inequality we used the first estimate of \eqref{e:firstlmlipa}. This gives (because $\rho<\eta r$)
\begin{equation}\label{e:firstlmlip2}
\frac{3}{4}|b(u_x,r)| \leq \aver{\partial B_{\rho}}|h_r| \leq \frac{5}{4}|b(u_x,r)|.
\end{equation}
On the other hand, we now fix some $\rho=\rho_x \in (\frac{\eta r}{2},\eta r)$ such that
\begin{equation}\label{e:firstlmlip6}
\int_{\partial B_\rho}|u_x-h_r| \leq \frac{2}{\eta r}\int_{\eta r/2}^{\eta r} ds \int_{\partial B_s}|u_x-h_r|.
\end{equation}
By Cauchy-Schwarz's inequality, Poincar\'{e}'s inequality and \eqref{e:quminux0} applied to the test function $h_r$, it follows that we have
\begin{align*}
\int_{\partial B_\rho}|u_x-h_r| &\leq \frac{2}{\eta r} \int_{B_{\eta r}}|u_x-h_r| \leq C(\eta r)^{\frac{d}{2}-1}\left(\int_{B_{\eta r}}|u_x-h_r|^2\right)^{1/2} \\
&\leq C(\eta r)^{\frac{d}{2}-1}\left(\int_{B_r}|u_x-h_r|^2\right)^{1/2} 
\leq C(\eta r)^{\frac{d}{2}-1}r\left(\int_{B_r}|\nabla(u_x-h_r)|^2\right)^{1/2} \\
&\leq C(\eta r)^{\frac{d}{2}-1}r\left(r^{\delta_{\text{\tiny\sc A}}}\int_{B_r}|\nabla h_r|^2 + r^d\right)^{1/2} \leq C\eta^{\frac{d}{2}-1}r^{d}\big(r^{\delta_{\text{\tiny\sc A}}}\omega(u_x,r)^2+1\big)^{1/2} \\
&\leq C\eta^{\frac{d}{2}-1}r^{d}(r^{\delta_{\text{\tiny\sc A}}/2}\omega(u_x,r)+1)
\end{align*}
In view of the two hypothesis in \eqref{e:firstlmlipa} we then get 
\begin{align}\label{e:firstlmlip3}
\nonumber\aver{\partial B_{\rho}}|u_x-h_r| &\leq C\eta^{-\frac{d}{2}}r(r_0^{\delta_{\text{\tiny\sc A}}/2}\omega(u_x,r)+1) 
\leq C\eta^{-\frac{d}{2}}r\,\omega(u_x,r)\Big(r_0^{\delta_{\text{\tiny\sc A}}/2}+\frac{1}{\kappa_1}\Big) \\
&\leq C\eta^{-\frac{d}{2}}\gamma^{-1}|b(u_x,r)|\Big(r_0^{\delta_{\text{\tiny\sc A}}/2}+\frac{1}{\kappa_1}\Big)
\leq \frac{1}{4}|b(u_x,r)|,
\end{align}
where the last inequality holds if we choose $r_0$ small enough and $\kappa_1>0$ large enough (both depending on $\eta$) such that
\begin{equation}\label{e:firstlmlip4}
C\gamma^{-1}\Big(r_0^{\delta_{\text{\tiny\sc A}}/2}+\frac{1}{\kappa_1}\Big)\leq \frac{1}{4}\eta^{d/2}.
\end{equation}
Now, using \eqref{e:firstlmlip2} and \eqref{e:firstlmlip3} we have
\[ b^+(u_x,\rho) =\aver{\partial B_\rho}|u_x| \leq \aver{\partial B_\rho}|h_r| + \aver{\partial B_\rho}|u_x-h_r| \leq \frac{3}{2}|b(u_x,r)|, \]
and, using also that $h_r$ keeps the same sign on $\partial B_\rho$ by \eqref{e:firstlmlip1}, we have
\begin{equation}\label{e:firstlmlip5}
|b(u_x,\rho)| \geq \bigg|\aver{\partial B_\rho}h_r\bigg| - \bigg|\aver{\partial B_\rho}(u_x-h_r)\bigg| \geq \aver{\partial B_\rho}|h_r|-\aver{\partial B_\rho}|u_x-h_r| \geq \frac{1}{2}|b(u_x,r)|.
\end{equation}
This proves the second inequality in \eqref{e:firstlmlipa}. Moreover, \eqref{e:firstlmlip3} and \eqref{e:firstlmlip1} imply that
\begin{align*}
|b(u_x,\rho)-b(u_x,r)| &\leq \bigg| b(u_x,\rho)-\aver{\partial B_\rho}h_r\bigg| + \bigg|\aver{\partial B_\rho}h_r-b(u_x,r)\bigg| \\
&\leq \aver{\partial B_\rho}|u_x-h_r| + \aver{\partial B_\rho}|h_r-b(u_x,r)|\leq \frac{1}{2}|b(u_x,r)|,
\end{align*}
which shows that $b(u_x,r)$ and $b(u_x,\rho)$ have the same sign.

For the first estimate in \eqref{e:firstlmlipb}, by \eqref{e:firstlmlip5} and the first hypothesis in \eqref{e:firstlmlipa}, we have
\begin{equation*}
\frac{1}{\rho}|b(u_x,\rho)|\geq \frac{1}{2\rho}|b(u_x,r)| \geq \frac{1}{2\eta r}|b(u_x,r)| \geq \frac{\gamma}{2\eta}(1+\omega(u_x,r)),
\end{equation*}
which using \eqref{e:estcontua} gives (notice that we assumed that $B_{\lambda_{\text{\tiny\sc A}} r}(x)\subset D$)
\begin{align*}
1+\rho^{\delta_{\text{\tiny\sc A}}/2}\omega(u_x,\rho) &\leq 1+\omega(u_x,\rho) \leq 1+C\bigg( \omega(u_x,r) + \log\frac{r}{\rho} \bigg) \\
&\leq C\big(1+\omega(u_x,r) + |\log\eta|\big) 
\leq C\big(1+|\log\eta|\big)(1+\omega(u_x,r)) \\
&\leq C\big(1+|\log\eta|\big)\frac{2\eta}{\gamma\rho}|b(u_x,\rho)|.
\end{align*}
Finally, observe that with $\eta$ small enough (and also $r_0$ small enough and $\kappa_1$ large enough so that \eqref{e:firstlmlip4} still holds) we have
\[ \tau \, C\big(1+|\log\eta|\big)\frac{2\eta}{\gamma} \leq 1. \]
This completes the proof.
\end{proof}

We continue with a self-improvement lemma whose strategy is similar to the one followed in the previous lemma, the main difference being that we now consider $u_x$ with different points $x$.

\begin{lm}\label{l:seclmlip}
There exist constants $r_0\in(0,1)$ and $\tau_0\geq 1$ with the following property: if $x\in D$, $\tau\geq \tau_0$ and $\rho\leq r_0$ satisfy $B_{\lambda_{\text{\tiny\sc A}}\rho}(x)\subset D$,
\begin{equation}\label{e:seclmlipa}
\tau(1+\rho^{\delta_{\text{\tiny\sc A}}/2}\omega(u_x,\rho))\leq \frac{1}{\rho}|b(u_x,\rho)|\qquad\text{and}\qquad b^+(u_x,\rho)\leq3|b(u_x,\rho)|,
\end{equation}
then for every $y\in B_{\eps\rho}(x)$, where $\eps=\tau^{-1/d}$, there exists $\rho_1\in(\frac{\eps\rho}{2},\eps\rho)$ such that 
\begin{equation}\label{e:seclmlipb}
2\tau(1+\rho_1^{\delta_{\text{\tiny\sc A}}/2}\omega(u_y,\rho_1))\leq \frac{1}{\rho_1}|b(u_y,\rho_1)|\qquad\text{and}\qquad b^+(u_y,\rho_1)\leq3|b(u_y,\rho_1)|.
\end{equation}
Moreover, $b(u_x,\rho)$ and $b(u_y,\rho_1)$ have the same sign.
\end{lm}

\begin{proof}
Firstly, if $\eps$ is small enough so that $\bar{\eps}:=2\lambda_{\text{\tiny\sc A}}^2\eps\leq 1/4$, then by standard estimates on harmonic functions (see \cite[Theorem 3.9]{gilbarg-trudinger-01}) $h_\rho=h_{x,\rho}$ satisfies
\begin{equation*}\label{e:seclmlip1}
\|\nabla h_\rho\|_{L^\infty(B_{\bar{\eps} \rho})} \leq \frac{C}{\rho}\|h_\rho\|_{L^\infty(B_{\rho/2})} \leq \frac{C}{\rho}\aver{\partial B_\rho}|h_\rho| = \frac{C}{\rho}b^+(u_x,\rho).
\end{equation*}
Using that $b(u_x,\rho)=b(h_\rho,\rho)=h_\rho(0)$ by harmonicity and the second hypothesis in \eqref{e:seclmlipa}, it follows that for every $\xi\in B_{\bar{\eps}\rho}$ we have
\begin{equation}\label{e:seclmlip2}
|b(u_x,\rho)-h_\rho(\xi)| \leq \bar{\eps}\rho\|\nabla h_\rho\|_{L^\infty(B_{\bar{\eps}\rho})} \leq \eps Cb^+(u_x,\rho) \leq \tau_0^{-\frac{1}{d}} C|b(u_x,\rho)|\leq \frac{1}{4}|b(u_x,\rho)|, 
\end{equation}
where the last inequality holds if $\tau_0$ is big enough. 
This implies that
\begin{equation}\label{e:seclmlip21}
\frac{3}{4}|b(u_x,\rho)| \leq |h_\rho(\xi)|\leq\frac{5}{4}|b(u_x,\rho)|\qquad\text{for every}\qquad \xi\in B_{\bar{\eps}\rho}.
\end{equation}
Moreover, by \eqref{e:quminux0} applied to $h_\rho$ (and since $\bar{\eps}\rho\leq \rho$ for $\tau_0$ large enough) we have
\begin{align}\label{e:seclmlip3}
\nonumber \int_{B_{\bar{\eps}\rho}}|u_x-h_\rho| &\leq C(\bar{\eps}\rho)^{\frac{d}{2}}\left(\int_{B_{\bar{\eps}\rho}}|u_x-h_\rho|^2\right)^{1/2} \leq C(\eps \rho)^{\frac{d}{2}}\left(\int_{B_\rho}|u_x-h_\rho|^2\right)^{1/2} \\
&\leq C(\eps \rho)^{\frac{d}{2}}\rho\left(\int_{B_\rho}|\nabla(u_x-h_\rho)|^2\right)^{1/2} \leq C(\eps\rho)^{\frac{d}{2}}\rho\left(\rho^{\delta_{\text{\tiny\sc A}}}\int_{B_\rho}|\nabla h_\rho|^2 + \rho^d\right)^{1/2} \\
\nonumber&\leq  C\eps^{\frac{d}{2}}\rho^{d+1}\big(\rho^{\delta_{\text{\tiny\sc A}}}\omega(u_x,\rho)^2+1\big)^{1/2} \leq C\eps^{\frac{d}{2}}\rho^{d+1}(\rho^{\delta_{\text{\tiny\sc A}}/2}\omega(u_x,\rho)+1).
\end{align}
We now fix some $y\in B_{\eps\rho}(x)$. Let $F:B_{\eps\rho}\subset \R^d\rightarrow B_{\bar{\eps}\rho}\subset\R^d$ be the function defined by $F(z)=F_x^{-1}\circ F_y(z)$. Then the coarea formula gives (and because $\partial F(B_s)=\partial \{|F^{-1}|>s\}$)
\begin{align}\label{e:seclmlip4}
\nonumber\lambda_{\text{\tiny\sc A}}^{-2}\int_{\eps\rho/2}^{\eps\rho}ds\int_{\partial F(B_s)}|u_x-h_\rho|\,d\mathcal{H}^{d-1} &\leq 
\int_{\eps\rho/2}^{\eps\rho}ds\int_{\partial F(B_s)}\frac{|u_x-h_\rho|}{|\nabla |F^{-1}||}\,d\mathcal{H}^{d-1} \\
&= \int_{\{\eps\rho/2\leq |F^{-1}|\leq\eps\rho\}}|u_x-h_\rho| \\
\nonumber&\leq \int_{F(B_{\eps\rho})}|u_x-h_\rho|
\leq \int_{B_{\bar{\eps} \rho}}|u_x-h_\rho|.
\end{align}
We now choose $\rho_1\in(\frac{\eps\rho}{2},\eps\rho)$ such that
\begin{equation*}\label{e:seclmlip5}
\int_{\partial F(B_{\rho_1})}|u_x-h_\rho|\,d\mathcal{H}^{d-1} \leq \frac{2}{\eps\rho} \int_{\eps\rho}^{\eps\rho/2}ds\int_{\partial F(B_{\rho_1})}|u_x-h_\rho|\,d\mathcal{H}^{d-1},
\end{equation*}
so that \eqref{e:seclmlip3}, \eqref{e:seclmlip4} and the first hypothesis in \eqref{e:seclmlipa} imply 
\begin{align}\label{e:seclmlip6}
\nonumber\aver{\partial F(B_{\rho_1})}|u_x-h_\rho|\,d\mathcal{H}^{d-1} &\leq C\eps^{-\frac{d}{2}}\rho(\rho^{\delta_{\text{\tiny\sc A}}/2}\omega(u_x,\rho)+1) \leq C\eps^{-\frac{d}{2}}\tau^{-1}|b(u_x,\rho)| \\
&\leq C\tau_0^{-1/2} |b(u_x,\rho)| \leq \frac{1}{4}\lambda_{\text{\tiny\sc A}}^{-4(d-1)}|b(u_x,\rho)|,
\end{align}
where the last inequality holds for $\tau_0$ is large enough.
Moreover, because the functions $F$ and $F^{-1}$ are Lipschitz continuous with Lipschitz constants bounded by $\lambda_{\text{\tiny\sc A}}^2$, we have for every set $E\subset \R^d$ (see \cite[Proposition 3.5]{maggi-12})
\begin{equation}\label{e:seclmlip8}
\lambda_{\text{\tiny\sc A}}^{-2(d-1)}\mathcal{H}^{d-1}(E) \leq F_{\#}\mathcal{H}^{d-1}(E) \leq \lambda_{\text{\tiny\sc A}}^{2(d-1)}\mathcal{H}^{d-1}(E),
\end{equation}
where $F_{\#}\mathcal{H}^{d-1}$ stands for the pushforward measure of $\mathcal{H}^{d-1}$ along $F$.
Therefore, by \eqref{e:seclmlip21} (and since $\partial F(B_{\rho_1})\subset B_{\bar{\eps}\rho}$), \eqref{e:seclmlip8} and \eqref{e:seclmlip6} we have
\begin{align}\label{e:seclmlip9}
\nonumber b^+(u_y,\rho_1) &=\aver{\partial B_{\rho_1}}|u_y|\,d\mathcal{H}^{d-1} = \frac{1}{\mathcal{H}^{d-1}(\partial B_{\rho_1})}\int_{\partial F(B_{\rho_1})}|u_x|\,dF_{\#}\mathcal{H}^{d-1} \\
\nonumber&\leq \frac{1}{\mathcal{H}^{d-1}(\partial B_{\rho_1})}\left(\int_{\partial F(B_{\rho_1})}|h_\rho|\,dF_{\#}\mathcal{H}^{d-1} + \int_{\partial F(B_{\rho_1})}|u_x-h_\rho|\,dF_{\#}\mathcal{H}^{d-1}\right) \\
&\leq \frac{5}{4}|b(u_x,\rho)| + \lambda_{\text{\tiny\sc A}}^{4(d-1)}\aver{\partial F(B_{\rho_1})}|u_x-h_\rho|\,d\mathcal{H}^{d-1} \\
\nonumber&\leq \frac{3}{2}|b(u_x,\rho)|.
\end{align}
On the other hand, we have by \eqref{e:seclmlip6}
\begin{align}\label{e:seclmlip11}
\nonumber\left|b(u_y,\rho_1) - \frac{1}{\mathcal{H}^{d-1}(\partial B_{\rho_1})}\int_{\partial F(B_{\rho_1})}h_\rho\,dF_{\#}\mathcal{H}^{d-1}\right| &\leq \frac{1}{\mathcal{H}^{d-1}(\partial B_{\rho_1})} \int_{\partial F(B_{\rho_1})}|u_x-h_\rho|\,dF_{\#}\mathcal{H}^{d-1} \\
&\leq \lambda_{\text{\tiny\sc A}}^{4(d-1)}\aver{\partial F(B_{\rho_1})}|u_x-h_\rho|\,d\mathcal{H}^{d-1} \\
\nonumber&\leq \frac{1}{4}|b(u_x,\rho)|.
\end{align}
Moreover, by \eqref{e:seclmlip2} and since $\partial F(B_{\rho_1}) \subset B_{\bar{\eps}\rho}$ we have
\begin{multline}\label{e:seclmlip12}
\left|\frac{1}{\mathcal{H}^{d-1}(\partial B_{\rho_1})}\int_{\partial F(B_{\rho_1})}h_\rho\,dF_{\#}\mathcal{H}^{d-1} - b(u_x,\rho)\right| \\ \leq \max\left\{ \max_{\xi\in\partial F(B_{\rho_1})}h_\rho(\xi) - b(u_x,\rho), \ b(u_x,\rho)-\min_{\xi\in\partial F(B_{\rho_1})}h_\rho(\xi)\right\} 
\leq \frac{1}{4}|b(u_x,\rho)|.
\end{multline}
Therefore, using the triangle inequality, \eqref{e:seclmlip11} and \eqref{e:seclmlip12} we get
\begin{equation*}
|b(u_y,\rho_1)-b(u_x,\rho)| \leq \frac{1}{2}|b(u_x,\rho)|.
\end{equation*}
This proves that $b(u_x,\rho)$ and $b(u_y,\rho_1)$ have the same sign and also implies that 
\begin{equation}\label{e:seclmlip10}
|b(u_y,\rho_1)| \geq |b(u_x,\rho)| - |b(u_y,\rho_1)-b(u_x,\rho)|\geq \frac{1}{2}|b(u_x,\rho)|
\end{equation}
Finally, \eqref{e:seclmlip9} and \eqref{e:seclmlip10} gives
\begin{equation*}
b^+(u_y,\rho_1) \leq \frac{3}{2}|b(u_x,\rho)| \leq 3|b(u_y,\rho_1)|,
\end{equation*}
which is the second inequality in \eqref{e:seclmlipb}.

We now prove the first inequality in \eqref{e:seclmlipb}. By \eqref{e:seclmlip10} and the first hypothesis in \eqref{e:seclmlipa} we have
\begin{equation}\label{e:seclmlip13}
|b(u_y,\rho_1)| \geq \frac{1}{2}|b(u_x,\rho)| \geq \frac{\tau\rho}{2}(1+\rho^{\delta_{\text{\tiny\sc A}}/2}\omega(u_x,\rho)),
\end{equation}
We then apply Lemma \ref{l:estcontu} (notice that we have $B_{\lambda_{\text{\tiny\sc A}}\rho}(x)\subset D$ and $2\lambda_{\text{\tiny\sc A}}^2\rho_1\leq\rho$) and eventually choose $\tau_0$ bigger (depending only on $d$ and $\delta_{\text{\tiny\sc A}}$) to get
\begin{align*}
1+\rho_1^{\delta_{\text{\tiny\sc A}}/2}\omega(u_y,\rho_1) &\leq 1 +\rho_1^{\delta_{\text{\tiny\sc A}}/2}\lambda_{\text{\tiny\sc A}}^2\omega(u_x,2\lambda_{\text{\tiny\sc A}}^2\rho_1) \leq 1+\rho_1^{\delta_{\text{\tiny\sc A}}/2}C(\omega(u_x,\rho)+\log(2\bar{\eps}\,{}^{-1})) \\
&\leq C(1+\rho_1^{\delta_{\text{\tiny\sc A}}/2}\omega(u_x,\rho)+\rho_1^{\delta_{\text{\tiny\sc A}}/2}\log(2\bar{\eps}\,{}^{-1}))\\
&\leq C(1+\rho^{\delta_{\text{\tiny\sc A}}/2}\omega(u_x,\rho)+\tau_0^{-\delta_{\text{\tiny\sc A}}/2d}\log(\tau_0)) \\
&\leq C(1+\rho^{\delta_{\text{\tiny\sc A}}/2}\omega(u_x,\rho)).
\end{align*}
Therefore, with \eqref{e:seclmlip13} this gives
\begin{equation*}
\frac{1}{\rho_1}|b(u_y,\rho_1)| \geq \frac{1}{\eps\rho}\frac{\tau\rho}{2}(1+\rho^{\delta_{\text{\tiny\sc A}}/2}\omega(u_x,\rho)) 
\geq \frac{\tau^{1/d}}{2C}\tau(1+\rho_1^{\delta_{\text{\tiny\sc A}}/2}\omega(u_y,\rho_1)),
\end{equation*}
which, choosing $\tau_0$ big enough so that $\tau_0^{1/d}\geq 4C$, completes the proof.
\end{proof}

We are now in position to prove Proposition \ref{p:lipcase1} using the results from Lemmas \ref{l:firstlmlip}, \ref{l:seclmlip} and Proposition \ref{p:strposlip}.

\begin{proof}[Proof of Proposition \ref{p:lipcase1}]
Set $\eps=\tau_0^{-1/d}$ and $\bar{r}=\frac{\eps\eta}{2}r$ where $\eta$ and $\tau_0$ are the constants given by Lemmas \ref{l:firstlmlip} and \ref{l:seclmlip}.
Note that in view of the first hypothesis in \eqref{e:lipcase1a} we have $b(u_x,r)\neq 0$.
We will prove that if $b(u_x,r)>0$ (resp. if $b(u_x,r)<0$), then $u>0$ almost everywhere (resp. $u<0$ a.e.) in $B_{\bar{r}}(x)$ .  

Let $y \in B_{\bar{r}}(x)$ be fixed. We first apply Lemma \ref{l:firstlmlip}. Now, we apply once Lemma \ref{l:seclmlip} at $x$ (notice that we have $y \in B_{\eps\rho}(x)$) and then iteratively at the point $y$. It follows that there exists a sequence of raddi $\rho_i>0$ such that
\begin{equation}\label{e:firstcaselip1}
2^i\tau(1+\rho_i^{\delta_{\text{\tiny\sc A}}/2}\omega(u_y,\rho_i))\leq\frac{1}{\rho_i}b(u_y,\rho_i),\qquad i\geq 0,
\end{equation}
and that $b(u_y,\rho_i)$ has the same sign than $b(u_x,r)$ for every $i\geq 0$.
Assume that $b(u_x,r)>0$, the proof in the case $b(u_x,r)<0$ is identical. Let us denote by $h_i=h_{y,\rho_i}$ the harmonic extension of the trace of the function $u_y$ to $\partial B_{\rho_i}$.
With the same argument as in \eqref{e:seclmlip2} we get
\begin{equation*}
|b(u_y,\rho_i)-h_i(\xi)| \leq \frac{1}{4}|b(u_y,\rho_i)|\qquad\text{for every }\xi\in B_{\eps\rho_i}.
\end{equation*}
Since $b(u_y,\rho_i)>0$, this implies that for every $\xi\in B_{\eps\rho_i}\cap \{ u_y\leq 0\}$ we have
\begin{align*}
|u_y(\xi)-h_i(\xi)| &\geq |u_y(\xi)-b(u_y,\rho_i)|-|b(u_y,\rho_i)-h_i(\xi)| \\
&\geq |b(u_y,\rho_i)| - \frac{1}{4}|b(u_y,\rho_i)| = \frac{3}{4}|b(u_y,\rho_i)|.
\end{align*}
By the Chebyshev inequality, the Lebesgue measure of $B_{\eps\rho_i}\cap \{ u_y\leq 0\}$ is estimate as
\begin{equation}\label{e:firstcaselip2}
|B_{\eps\rho_i}\cap \{ u_y\leq 0\}| \leq \frac{4}{3|b(u_y,\rho_i)|}\int_{B_{\eps\rho_i}}|u_y-h_i|.
\end{equation}
On the other hand, by \eqref{e:seclmlip3} in this context we have
\begin{equation}\label{e:firstcaselip3}
\int_{B_{\eps\rho_i}}|u_y-h_i| \leq C\eps^{d/2}\rho_i^{d+1}(1+\rho_i^{\delta_{\text{\tiny\sc A}}/2}\omega(u_y,\rho_i)).
\end{equation}
Now, combining \eqref{e:firstcaselip2}, \eqref{e:firstcaselip3} and \eqref{e:firstcaselip1} we get
\[ \frac{|B_{\eps\rho_i}\cap \{ u_y\leq 0\}|}{|B_{\eps\rho_i}|} \leq (\eps\rho_i)^{-d}C\eps^{d/2}\rho_i^{d+1}(2^i\tau\rho_i)^{-1} \leq \eps^{d/2}C2^{-i}, \]
which implies that
\[ \frac{|F_y(B_{\eps\rho_i})\cap\{u\leq 0\}|}{|F_y(B_{\eps\rho_i})|} = \frac{|B_{\eps\rho_i}\cap \{ u_y\leq 0\}|}{|B_{\eps\rho_i}|}\xrightarrow[i\to+\infty]{} 0, \]
where $F_y(B_{\eps\rho_i})=y+\rho_iA_y^{\sfrac12}(B_\eps)$. This shows that the density of the set $\{u\leq 0\}$ at every point $y\in B_{\bar{r}}(x)$ is $0$ (see \cite[exercise 5.19]{maggi-12}), and hence that $u>0$ almost-everywhere in $B_{\bar{r}}(x)$. 
Now, we set $c=\lambda_{\text{\tiny\sc A}}^{-1}\tau_0^{-1/d}\eta/2$, where $\eta$ and $\tau_0$ are the constants given by Lemma \ref{l:firstlmlip} and \ref{l:seclmlip}.
Then \eqref{e:lipcase1b} and \eqref{e:lipcase1c} follow from Proposition \ref{p:strposlip} and the fact that $\omega(u,x,cr)\leq c^{-d/2}\omega(u,x,r)$.
This concludes the proof.
\end{proof}

\subsection{Lipschitz continuity}\label{sub:endlipcont}
In this subsection we prove Theorem \ref{t:main2}. At this stage, the main work is to deal with the case where $\omega(u_x,r)$ is big and the first condition in \eqref{e:lipcase1a} fails. In this case, we show in Proposition \ref{p:case2} below that the value of $\omega(u_x,r)$ decreases at some smaller scale. Notice that the extra hypothesis \eqref{e:case2a} is almost irrelevant in view of Proposition \ref{p:strposlip}. 

\begin{prop}\label{p:case2}
Let $K\subset D$ be a compact set. There exist positive constants $r_K, \gamma\in(0,1)$ and $\kappa_2>0$ with the following property: if $x\in K$ and $r\leq r_K$ satisfy
\begin{equation}\label{e:case2a}
u_x(\xi_0)=0\quad\text{for some}\quad \xi_0\in B_{{\scriptscriptstyle(2\lambda_{\scalebox{.5}{A}})^{-2}}r},
\end{equation}
\begin{equation}\label{e:case2b}
|b(u_x,r)|\leq \gamma r(1+\omega(u_x,r))\qquad\text{and}\qquad \kappa_2\leq\omega(u_x,r),
\end{equation}
then we have
\begin{equation}\label{e:case2c}
\omega(u_x,r/3) \leq \frac{1}{2}\omega(u_x,r).
\end{equation}
\end{prop}

We will need the following almost-monotonicity formula for operators in divergence form. We refer to  \cite[Theorem III]{matevosyan-petrosyan-11} for a proof (see also \cite{alt-caffarelli-friedman-84} and \cite{caffarelli-jerison-kenig-04} for the case of the Laplacian). Let us set for $u_+,u_-\in H^1(B_1)$ and $r\in(0,1)$
\[ \Phi(u_+,u_-,r)=\left(\frac{1}{r^2}\int_{B_r}\frac{|\nabla u_+(\xi)|^2}{|\xi|^{d-2}}\,d\xi\right)\left(\frac{1}{r^2}\int_{B_r}\frac{|\nabla u_-(\xi)|^2}{|\xi|^{d-2}}\,d\xi\right). \]

\begin{prop}\label{p:monotony}
Let $B=(b_{ij})_{ij} : B_1\rightarrow \text{Sym}_d^+$ be a uniformly elliptic matrix-valued function with H\"{o}lder continuous coefficients, that is, for every $x,y\in B_1$ and $\xi\in\R^d$
\begin{equation*}
|b_{ij}(x)-b_{ij}(y)|\le c_{\scalebox{.5}{B}}|x-y|^{\delta_{\scalebox{.5}{B}}}\qquad\text{and}\qquad
\frac{1}{\lambda_{\scalebox{.5}{B}}^2}|\xi|^2\le \xi\cdot B_x\,\xi\le \lambda_{\scalebox{.5}{B}}^2|\xi|^2.
\end{equation*} 
Let $u_+,u_-$ be two non-negative and continuous functions in the unit ball $B_1$ such that
\[ \dive(B\nabla u_\pm) \geq -1\quad\text{in}\quad B_1\qquad\text{and}\qquad u_+u_-=0\quad\text{in}\quad B_1. \]
Then there exist $r_0>0$ and $C>0$, depending only on $d, c_{\scalebox{.5}{B}}, \delta_{\scalebox{.5}{B}}$ and $\lambda_{\scalebox{.5}{B}}$, such that for every $r\leq r_0$ we have
\[ \Phi(u_+,u_-,r) \leq C\big(1+\|u_++u_-\|_{L^2(B_1)}^2\big)^2. \]
\end{prop}

We now state this almost-monotonicity formula for the functions $u_x^\pm$.

\begin{coro}\label{c:monotony}
Let $\O\subset D$ be a quasi-open set and $K\subset D$ be a compact set. Let $A$ be a matrix-valued function satisfying \eqref{e:holderA} and \eqref{e:ellipA}. Let $f\in L^\infty(D)$. Assume that $u\in H^1_0(\O)$ is a continuous function solution of the equation
\begin{equation}\label{e:monotoni}
-\dive(A\nabla u) = f \quad\text{in}\quad\O.
\end{equation}
Then there exists $r_K>0$ and $C_m>0$, depending only on $d, c_{\text{\tiny\sc A}}, \delta_{\text{\tiny\sc A}}, \lambda_{\text{\tiny\sc A}}, \|f\|_{L^\infty}, |D|$ and $\text{dist}(K,D^c)$, such that for every $x\in K$ and every $r\leq r_K$ the function $u_x$ satisfies
\[ \Phi(u_x^+,u_x^-,r) \leq C_m. \]
\end{coro}

\begin{proof}
We first prove that we have, in the sense of distributions,
\begin{equation}\label{e:monotoni1}
\dive(A\nabla u^+)\geq f\ind_{\{u>0\}}\quad\text{in}\quad D\qquad\text{ and }\qquad \dive(A\nabla u^-)\geq f\ind_{\{u<0\}}\quad\text{in}\quad D.
\end{equation}
Let us define $p_n : \R \rightarrow \R^+$ for $n\in\N$ by
\[p_n(s)=0,\  \text{ for }\ s \leq 0;\qquad p_n(s) = ns,\  \text{ for }\  s \in [0,1/n]; \qquad p_n(s) =1,\  \text{ for }\  s \geq 1/n, \]
and set $q_n(s)=\int_0^s p_n(t)\,dt$.
Since $p_n$ is Lipschitz continuous, we have $p_n(u) \in H^1_0(\O)$ and $\nabla p_n(u) = p'_n(u)\nabla u$. Let $\varphi \in C^{\infty}_0(D)$ be such that $\varphi\ge 0$ in $D$. Multiplying the equation \eqref{e:monotoni} with $\varphi p_n(u)\in H^1_0(\O)$ we get
\begin{align*}
\int_D A\nabla q_n(u)\cdot\nabla\varphi =\int_D p_n(u) A\nabla u \cdot \nabla \varphi &\le \int_D\big(p_n(u) A\nabla u \cdot \nabla \varphi + \varphi p'_n(u) A\nabla u\cdot\nabla u) \\
&= \int_D A\nabla u\cdot\nabla(\varphi p_n(u)) =\int_D f \varphi p_n(u).
\end{align*}
Now, the inequality for $u^+$ in \eqref{e:monotoni1} follows by letting $n$ tend $+\infty$, because $p_n(u)$ converges almost-everywhere to $\ind_{\{u>0\}}$ and $\nabla q_n(u)$ converges in $L^2$ to $\nabla u^+$. The same proof holds for $u^-$.

Now, set $\rho=\lambda_{\text{\tiny\sc A}}^{-1}\text{dist}(K,D^c)$ and define for every $\xi\in B_1$ 
\[ u_\pm(\xi)=\rho^{-2}\|f\|_{L^\infty}^{-1}\,u_x^\pm(\rho \xi),\qquad \tilde{f}(\xi)=f\circ F_x(\rho \xi),\qquad B_\xi=A_x^{-\sfrac12}A_{F_x(\rho \xi)}A_x^{-\sfrac12}. \]
Then the functions $u_\pm$ satisfy
\[ \dive(B\nabla u_\pm)\geq \|f\|_{L^\infty}^{-1}\tilde{f}\ind_{\{u_\pm>0\}}\geq -1\quad\text{in}\quad B_1. \]
Therefore, by Proposition \ref{p:monotony} we have for every $r\leq r_K:=r_0\rho$
\begin{align*}
\Phi(u_x^+,u_x^-,r) &= \rho^4\|f\|_{L^\infty}^4\Phi(u_+,u_-,r/\rho) \leq \rho^4\|f\|_{L^\infty}^4C\big(1+\|u_++u_-\|_{L^2(B_1)}^2\big)^2 \\
&\leq \rho^4\|f\|_{L^\infty}^4C\Big(1+\lambda_{\text{\tiny\sc A}}^d \rho^{-d-4}\|f\|_{L^\infty}^{-2}\|u\|_{L^2(D)}^2\Big)^2  \\
&\leq \rho^4\|f\|_{L^\infty}^4C\Big(1+\lambda_{\text{\tiny\sc A}}^d\rho^{-d-4}C_d|D|^{1+4/d}\Big)^2 =: C_m.
\end{align*}
\end{proof}

\begin{proof}[Proof of Proposition \ref{p:case2}]
Let us denote as before $h_r=h_{x,r}$ the harmonic extension of the trace of $u_x$ to $\partial B_r$. Then we have
\begin{equation}\label{e:case24}
\omega(u_x,r/3)^2 = \aver{B_{r/3}}|\nabla u_x|^2 \leq 2\,\aver{B_{r/3}}|\nabla h_r|^2 + 2\,\aver{B_{r/3}}|\nabla (u_x-h_r)|^2.
\end{equation}
By the quasi-minimality property of $u_x$ we can estimate the second term in the right hand side of \eqref{e:case24} as we did in \eqref{e:estcontu1}, this gives
\begin{align}\label{e:case25}
\nonumber\aver{B_{r/3}}|\nabla (u_x-h_r)|^2 \leq 3^d\aver{B_r}|\nabla (u_x-h_r)|^2 &\leq Cr^{\delta_{\text{\tiny\sc A}}}\omega(u_x,r)^2 +C\\
&\leq C(r^{\delta_{\text{\tiny\sc A}}} + \kappa_2^{-2})\,\omega(u_x,r)^2,
\end{align}
where in the last inequality we have used the second hypothesis in \eqref{e:case2b}. On the other hand, estimates for harmonic functions give
\begin{equation}\label{e:case26}
\|\nabla h_r\|_{L^\infty(B_{r/3})} \leq \frac{C}{r}\|h_r\|_{L^\infty(B_{r/2})} \leq \frac{C}{r}\aver{\partial B_r}|h_r| = \frac{C}{r}b^+(u_x,r).
\end{equation}
We now want to estimate $b^+(u_x,r)$ in terms of $r\,\omega(u_x,r)$. Let us assume that $\omega(u_x^+,r)\leq\omega(u_x^-,r)$, the same proof holds if the opposite inequality is satisfied.
We first prove that for $\xi_0 \in B_{r/2}$ and $\eta< 1/2$ we have
\begin{equation}\label{e:case21}
\aver{\partial B_r}u_x^+ - \aver{\partial B_{\eta r}(\xi_0)}u_x^+ \leq C_d \eta^{1-d}r\,\aver{B_r}|\nabla u_x^+|.
\end{equation}
Notice that up to considering the function $\xi\mapsto u_x^+(r\xi)$ we can assume that $r=1$. Let us define a one to one function $F:B_1\setminus B_\eta \rightarrow B_1\setminus B_\eta(\xi_0)$ by
\begin{equation*}
F(\xi) = \xi + \frac{1-|\xi|}{1-\eta}\xi_0.
\end{equation*}
We set $v=u_x^+\circ F$. For every $\xi \in\partial B_1$ we have by the fundamental theorem of the calculus 
\begin{align*}
v(\xi)-v(\eta\xi) = \int_{\eta}^1\frac{d}{dt}v(t\xi)\,dt \leq \int_{\eta}^1|\nabla v(t\xi)|dt.
\end{align*} 
Note that $F$ is the identity on $\partial B_1$ and is simply a translation on $\partial B_\eta$. Therefore, averaging on $\xi\in\partial B_1$ (and since $|\nabla v|\leq C_d|\nabla u_x^+\circ F|$) we get
\begin{align*}
\aver{\partial B_1}u_x^+ - \aver{\partial B_\eta(\xi_0)}u_x^+ &= \aver{\partial B_1}v - \aver{\partial B_\eta(\xi_0)}v \leq \frac{1}{d\omega_d}\int_{\eta}^{1}dt\int_{\partial B_1}|\nabla v(t\xi)|\,d\mathcal{H}^{d-1}(\xi) \\
&\leq \frac{\eta^{1-d}}{d\omega_d}\int_{B_1\setminus B_\eta}|\nabla v| \leq C_d\eta^{1-d} \int_{B_1\setminus B_\eta}|\nabla u_x^+\circ F| \\
&\leq C_d\eta^{1-d} \int_{B_1\setminus B_\eta(\xi_0)}|\nabla u_x^+||\det\nabla F^{-1}| \leq C_d\eta^{1-d} \aver{B_1}|\nabla u_x^+|,
\end{align*}
which proves \eqref{e:case21}.
Now, let $\xi_0\in B_{{(2\scriptscriptstyle\lambda_{\text{\tiny\sc A}})^{-2}}r}$ be such that $u_x(\xi_0)=0$ as in \eqref{e:case2a}.
By Proposition \ref{p:contu} we have for every $\xi \in B_{\eta r}(\xi_0)$ (and because $F_x(\xi_0),F_x(\xi)\in B_{{(2\scriptscriptstyle\lambda_{\text{\tiny\sc A}})^{-1}}r}(x)$ if $\eta$ is small enough)
\begin{align}\label{e:case22}
\nonumber u_x^+(\xi)&\leq|u_x(\xi)| = |u_x(\xi_0)-u_x(\xi)|=|u(F_x(\xi_0))-u_x(F_x(\xi))| \\
\nonumber &\leq C|F_x(\xi_0)-F_x(\xi)|\bigg(1+\omega(u,x,\lambda_{\text{\tiny\sc A}}^{-1}r) + \log\frac{\lambda_{\text{\tiny\sc A}}^{-1}r}{|F_x(\xi_0)-F_x(\xi)|}\bigg) \\
&\leq C|\xi_0-\xi|\bigg(1+\omega(u,x,\lambda_{\text{\tiny\sc A}}^{-1}r) + \log\frac{r}{|\xi_0-\xi|}\bigg) \\
\nonumber &\leq C r(1+\eta\,\omega(u_x,r)),
\end{align}
where the last inequality holds for $\eta$ small enough and since we have
\begin{align*}
\omega(u,x,\lambda_{\text{\tiny\sc A}}^{-1}r)^2 =\aver{B_{{\scriptscriptstyle\lambda_{\text{\tiny\sc A}}^{-1}}r}(x)}|\nabla u|^2 &\leq \lambda_{\text{\tiny\sc A}}^2\aver{B_{{\scriptscriptstyle\lambda_{\text{\tiny\sc A}}^{-1}}r}(x)}A_x\nabla u\cdot\nabla u  \\
&=\lambda_{\text{\tiny\sc A}}^2\aver{F_x^{-1}(B_{{\scriptscriptstyle\lambda_{\text{\tiny\sc A}}^{-1}}r}(x))}|\nabla u_x|^2 \leq \lambda_{\text{\tiny\sc A}}^{2(d+1)}\omega(u_x,r)^2.
\end{align*}
Moreover, recall that we assumed that $\omega(u_x^+,r)\leq\omega(u_x^-,r)$. Using the monotonicity formula in Corollary \ref{c:monotony} we get
\begin{equation*}
\omega(u_x^+,r)^4\leq\omega(u_x^+,r)^2\omega(u_x^-,r)^2 \leq C_d\Phi(u_x^+,u_x^-,r) \leq C_dC_m,
\end{equation*}
which implies by Cauchy-Schwarz's inequality
\begin{equation}\label{e:case23}
\aver{B_r}|\nabla u_x^+| \leq \bigg(\aver{B_r}|\nabla u_x^+|^2\bigg)^{1/2} = \omega(u_x^+,r) \leq (C_dC_m)^{1/4}.
\end{equation}
Therefore, combining \eqref{e:case22}, \eqref{e:case21}, \eqref{e:case23} and using the first hypothesis in \eqref{e:case2b} we have (and also since $u_x^-=u_x^+-u_x$)
\begin{align}\label{e:case27}
\nonumber b^+(u_x,r) &=\aver{\partial B_r}|u_x| = 2\,\aver{\partial B_r}u_x^+ - \aver{\partial B_r}u_x \\
\nonumber&\leq 2\|u_x^+\|_{L^\infty(\partial B_{\eta r}(\xi_0))} + 2\bigg(\aver{\partial B_r}u_x^+ - \aver{\partial B_{\eta r}(\xi_0)}u_x^+\bigg) + |b(u_x,r)| \\
&\leq C\left((\eta+\gamma)\,\omega(u_x,r) +1+\eta^{1-d}C_m^{1/4} \right)r \\
\nonumber&\leq C\left((\eta+\gamma)+\big(1+\eta^{1-d}C_m^{1/4}\big)\kappa_2^{-1} \right)r\omega(u_x,r),
\end{align}
where in the last inequality we used the second hypothesis in \eqref{e:case2b}.
We now return to \eqref{e:case24}. With \eqref{e:case25}, \eqref{e:case26} and \eqref{e:case27} we get
\begin{align*}
\omega(u_x,r/3)^2 &\leq \frac{C}{r^2}b^+(u_x,r)^2+C(r^{\delta_{\text{\tiny\sc A}}}+\kappa_2^{-2})\,\omega(u_x,r)^2 \\
&\leq C\left((\eta+\gamma)^2+r^{\delta_{\text{\tiny\sc A}}}+\big(1+\eta^{1-d}C_m^{1/4}\big)^2\kappa_2^{-2} \right)\omega(u_x,r)^2.
\end{align*}
Therefore, choosing first $\eta, \gamma$ and $r_K$ small enough and then $\kappa_2$ big enough (depending on $\eta$) we obtain \eqref{e:case2c}, which concludes the proof.
\end{proof}

We are now in position to prove Theorem \ref{t:main2} using an iterative argument and Propositions \ref{p:strposlip}, \ref{p:lipcase1}, \ref{p:case2}.

\begin{proof}[Proof of Theorem \ref{t:main2}]
Recall that we denote by $u$ any coordinate function of the vector $U$ and that we have to prove that $u$ is locally Lipschitz continuous in $D$. Let $K\subset D$ be a compact set and let $x\in K$. Let $r\leq r_K$, where $r_K$ is smaller than the constants given by Propositions \ref{p:strposlip}, \ref{p:lipcase1} and \ref{p:case2}. Set $\kappa=\max\{\kappa_1,\kappa_2\}$ where $\kappa_1,\kappa_2$ are the constants given by Propositions \ref{p:lipcase1} and \ref{p:case2}. We consider the following four cases:

\noindent{\bf Case 1:}
\begin{equation}\label{e:case1}
\text{either}\qquad u_x>0\quad\text{in}\quad B_{{(2\scriptscriptstyle\lambda_{\text{\tiny\sc A}})^{-2}r}}\qquad\text{ or }\qquad u_x<0\quad\text{in}\quad B_{{(2\scriptscriptstyle\lambda_{\text{\tiny\sc A}})^{-2}r}}
\end{equation}
\noindent{\bf Case 2:}
\begin{equation}\label{e:case2}
\gamma r(1+\omega(u_x,r)) \leq |b(u_x,r)|\qquad\text{and}\qquad \kappa\leq\omega(u_x,r),
\end{equation}
\noindent{\bf Case 3:}
\begin{equation}\label{e:case3}
|b(u_x,r)|\leq \gamma r(1+\omega(u_x,r))\qquad\text{and}\qquad \kappa\leq\omega(u_x,r),
\end{equation}
\noindent{\bf Case 4:}
\begin{equation}\label{e:case4}
\omega(u_x,r)\leq\kappa.
\end{equation}
For $k\geq 0$ we set $r_k=3^{-k}r$. 
We denote by $k_0$, if it exists, the smallest integer $k\geq 0$ such that the pair $(x,r_k)$ satisfies either \eqref{e:case1} or \eqref{e:case2}, and we set $k_0=+\infty$ otherwise. If $k_0>0$, then for every $k<k_0$ we have that: if $(x,r_k)$ satisfies \eqref{e:case3} then by Proposition \ref{p:case2} we have (notice that \eqref{e:case2a} holds since $u$ is continuous and that \eqref{e:case1} is not satisfied)
\begin{equation*}
\omega(u_x,r_{k+1})\leq\frac{1}{2}\omega(u_x,r_k),
\end{equation*}
while if $(x,r_k)$ satisfies \eqref{e:case4}, then we have
\begin{equation*}
\omega(u_x,r_{k+1}) \leq 3^{d/2}\omega(u_x,r_k)\leq 3^{d/2}\kappa.
\end{equation*}
Therefore, with an induction we get that for every $0\leq k\leq k_0$
\begin{equation}\label{e:main11}
\omega(u_x,r_k) \leq \max\{2^{-k}\omega(u_x,r),3^{d/2}\kappa \}.
\end{equation}
Assume that $k_0=+\infty$. If $x$ is a Lebesgue point for $\nabla u$, then $0$ is a Lebesgue point of $u_x$ and it follows from \eqref{e:main11} that
\begin{equation*}
|\nabla u(x)|\leq\lambda_{\text{\tiny\sc A}}|\nabla u_x(0)| =\lambda_{\text{\tiny\sc A}} \lim_{k\to+\infty}\omega(u_x,r_k) \leq \lambda_{\text{\tiny\sc A}} 3^{d/2} \kappa.
\end{equation*}
Assume now that $k_0<+\infty$. Then, by definition of $k_0$, the pair $(x,r_{k_0})$ satisfies either \eqref{e:case1} or \eqref{e:case2}. If \eqref{e:case1} holds, then Proposition \ref{p:strposlip} infers that $u$ is $C^{1,\beta}$ near $x$ and that we have (using also \eqref{e:main11})
\begin{align*}
|\nabla u(x)| &\leq C_K(1+\omega(u,x,(2\lambda_{\text{\tiny\sc A}})^{-2}r_{k_0})) \leq C_K(1+\omega(u,x,\lambda_{\text{\tiny\sc A}}^{-1}r_{k_0}))\\
&\leq C_K(1+\omega(u_x,r_{k_0})) \leq C_K(1+\max\{2^{-k_0}\omega(u_x,r),3^{d/2}\kappa \}) \\
&\leq C_K(1+\omega(u_x,r)).
\end{align*}
Moreover, by Proposition \ref{p:lipcase1} the same estimate holds if the pair $(x,r_{k_0})$ satisfies \eqref{e:case2}. Therefore, in all cases it follows that for almost every point $x\in K$ and every $r\leq r_K$ we have
\begin{equation}\label{e:main12}
|\nabla u(x)|\leq C_K(1+\omega(u_x,r)).
\end{equation}
Let now $x_0\in K$. Then, for almost every $x \in B_{r_K/2}(x_0)$, it follows by \eqref{e:main12} that 
\begin{align*}
|\nabla u(x)|&\leq C_K(1+\omega(u_x,r_K/2)) \leq C_K(1+\omega(u_{x_0},r_K)).
\end{align*}
With a compactness argument this proves that $u$ is locally Lipschitz continuous in $D$ and completes the proof.
\end{proof}

\section{Lipschitz continuity of the eigenfunctions}
This section is dedicated to the proof of Theorem \ref{t:main1}. Precisely, we prove that the vector $U=(u_1,\dots,u_k)$ of the first $k$ eigenfunctions on an optimal set for \eqref{e:shapeopt} is locally Lipschitz continuous in $D$. Using an idea taken from \cite{mazzoleni-terracini-velichkov-17}, we show that $U$ is a quasi-minimizer in the sense of \eqref{e:main2b}, and we then apply Theorem \ref{t:main2} to get the Lipschitz continuity of $U$.

\subsection{Preliminaries and existence of an optimal set}
We start with some properties about the spectrum of the operator in divergence form defined in \ref{e:defoperator}, and we then prove that the problem \eqref{e:shapeopt} admits a solution among the class of quasi-open sets. 
 
Let us define the weighted Lebesgue measure $m=b\,dx$, where $dx$ stands for the Lebesgue measure in $\R^d$. For a quasi-open set $\O\subset \R^d$, we define the spaces $L^2(\O;m)=L^2(\O)$ and $H^1_0(\O;m)=H^1_0(\O)$ endowed respectively with the following norms
\begin{equation*}
\|u\|_{L^2(\O;m)}=\left(\int_{\O}u^2\,dm\right)^{1/2}\qquad\text{and}\qquad \|u\|_{H^1(\O;m)}=\|u\|_{L^2(\O;m)} + \|\nabla u\|_{L^2(\O)}.
\end{equation*}
Moreover, if $\O=\R^d$ we will simply write $\|u\|_{L^2(m)}=\|u\|_{L^2(\R^d;m)}$.
We notice that, by the hypothesis \eqref{e:hypfctb} on the function $b$, the norms $\|\cdot\|_{L^2(\O;m)}$ and $\|\cdot\|_{L^2(\O)}$ are equivalent. We stress out that the choice of these norms is natural in view of \eqref{e:defoperator} and is motivated by the variational formulation of the sum of the first $k$ eigenfunctions (see \eqref{e:carsum} below).
Now, the Lax-Milgram theorem and the Poincar\'{e} inequality imply that for every $f \in L^2(\O,m)$ there exists a unique solution $u \in H^1_0(\O,m)$ to the problem
\[ -\dive(A\nabla u) = f\,b \text{ in } \O, \quad u \in H^1_0(\O,m). \]
The resolvent operator $R_{\O} : f \in L^2(\O;m)\rightarrow H^1_0(\O;m) \subset L^2(\O;m)$ defined as $R_{\O}(f)=u$ is a continuous, self-adjoint and positive operator. Since the Sobolev space $H^1_0(\O;m)$ is compactly embedded into $L^2(\O;m)$ (because we have assumed that $b\geq c_b>0$, see \eqref{e:hypfctb}), the resolvent $R_{\O}$ is in addition a compact operator. We say that a complex number $\lambda$ is an eigenvalue of the operator \eqref{e:defoperator} in $\O$ if there exists a non-trivial eigenfunction $u\in H^1_0(\O;m)$ solution of the equation
\[ -\dive(A\nabla u) = \lambda u\,b \text{ in } \O, \quad u\in H^1_0(\O;m).\]
The above properties of the resolvent ensure that the spectrum of the operator \eqref{e:defoperator} in $\O$ is given by an increasing sequence of eigenvalues which are strictly positive real numbers, non-necessarily distinct, and which we denote by
\[ 0< \lambda_1(\O) \leq \lambda_2(\O) \leq \cdots \leq \lambda_k(\O) \leq \cdots \]
The eigenvalues $\lambda_k(\O)$ are variationnaly characterized by the following min-max formula
\[ \lambda_k(\O) = \min_{\substack{V \text{ subspace of} \\ \text{dimension k of } H^1_0(\O,m)}} \ \max_{v\in V \backslash\{0\}} \ \frac{\int_{\O}A\nabla v\cdot\nabla v\,dx}{\int_{\O}v^2\,dm}.\]
Moreover, we denote by $u_k$ the normalized (with respect to the norm $\|\cdot\|_{L^2(\O;m)}$) eigenfunctions corresponding to the eigenvalues $\lambda_k(\O)$ and note that the family $(u_k)_k$ form an orthonormal system in $L^2(\O;m)$, that is
\begin{equation*} \int_{\O}u_iu_j\,dm=\delta_{ij}:=\left\{
\begin{aligned}
1 &\quad\text{if} &i=j, \\
0 &\quad\text{if} &i\neq j. \\
\end{aligned}\right. 
\end{equation*}
As a consequence, we have the following variational formulation for the sum of the first $k$ eigenvalues on a quasi-open set $\O$
\begin{equation}\label{e:carsum}
\sum_{i=1}^k\lambda_i(\O)=\min\Big\{ \int_{\O}A\nabla V\cdot\nabla V\,dx \ : \ V=(v_1,\dots,v_k)\in H^1_0(\O,\R^k), \ \int_{\O}v_iv_j\,dm=\delta_{ij} \Big\},
\end{equation}
for which the minimum is attained for the vector $U=(u_1,\dots,u_k)$. We now deduce from this characterization that the minimum in \eqref{e:shapeopt} is reached.

\begin{prop}[Existence]
The shape optimization problem \eqref{e:shapeopt} has a solution.
\end{prop}

\begin{proof}
Let $(\O_n)_{n\in \N}$ be a minimizing sequence of quasi-open sets to the problem \eqref{e:shapeopt} and denote by $U_n=(u_1^n,\dots,u_k^n)$ the first $k$ eigenfunctions on $\O_n$. Since the matrices $A_x$, $x\in D$, are uniformly elliptic, we have the following inequality
\[ \lambda_{\text{\tiny\sc A}}^{-2} \int_D |\nabla U_n|^2\,dx \leq \int_D A\nabla U_n\cdot\nabla U_n\,dx = \sum_{i=1}^k\lambda_i(\O_n) \]
which infers that the norm $\|U_n\|_{H^1}$ is uniformly bounded. Therefore, up to a subsequence, $U_n$ converges weakly in $H^1(D,\R^k)$ and strongly in $L^2(D,\R^k)$ to some $V\in H^1(D,\R^k)$. Notice that $V$ is an orthonormal vector.
Set $\O^\ast:=\{|V|> 0\}$. Then using \eqref{e:carsum}, the weak convergence in $H^1$ of $U_n$ to $V$ and the semi-continuity of the Lebesgue measure we have
\begin{align*}
\sum_{i=1}^k\lambda_i(\O^\ast) + \Lambda|\O^\ast| &\leq \int_D A\nabla V\cdot\nabla V\,dx + \Lambda|\{|V|>0\}| \\
&\leq \liminf_n \Big( \int_D A\nabla U_n\cdot\nabla U_n\,dx + \Lambda|\{|U_n|>0\}| \Big) \\
&\leq \liminf_n \Big( \sum_{i=1}^k \lambda_i(\O_n) + \Lambda|\O_n| \Big)
\end{align*}
which concludes the proof.
\end{proof}

In the next Lemma we prove that the eigenfunctions are bounded. This result is a consequence of Lemma \ref{l:degiorgi} and we refer to \cite[Lemma 5.4]{russ-trey-velichkov-19} for a proof which is based on an interpolation argument.

\begin{lm}[Boundedness of the eigenfunctions]\label{l:boundedness}
Let $\O\subset\R^d$ be a bounded quasi-open set. There exist a dimensional constant $n\in \N$ and a constant $C>0$ depending only on $d, \lambda_{\text{\tiny\sc A}}, c_b$ and $|\Omega|$, such that the resolvent operator $R_{\O}:L^2(\O;m)\rightarrow L^2(\O;m)$ satisfies
$$R^n (L^2(\Omega;m))\subset L^\infty(\Omega)\qquad \hbox{and}\qquad \|R^n\|_{\mathcal{L}(L^2(\Omega;m);L^\infty(\Omega))}\le C.$$
In particular, if $u$ is an eigenfunction on $\Omega$ normalized by $\left\Vert u\right\Vert_{L^2(m)}=1$, then $u\in L^\infty(\Omega)$ and 
$$\|u\|_{L^\infty}\le C\lambda(\Omega)^n,$$  
where $\lambda(\O)$ denotes the eigenvalue corresponding to $u$.
\end{lm}

To conclude this subsection, we show that the first eigenfunction on an optimal set $\O^\ast$ keeps the same sign on every connected component of $\O^\ast$. Notice that $\O^\ast$ may not be connected and has at most $k$ connected components.

\begin{lm}[Sign of the principal eigenfunction]
Let $\O\subset D$ be an open and connected set and let $u\in H^1_0(\O)$ be the normalized first eigenvalue on $\O$, that is
\[ -\dive(A\nabla u) = \lambda_1(\O)\,b\,u \quad\text{in}\quad \O\qquad\text{and}\qquad \int_{\O}u^2\,dm=1. \]
Then $u$ is non-negative in $\O$ (up to a change of sign).
\end{lm}

\begin{proof}
We assume that $u^+\neq 0$ (if not, take $-u$ instead of $u$) and we set
\[ u_+=u^+/\|u^+\|_{L^2(m)}\qquad \text{and}\qquad u_-=u^-/\|u^-\|_{L^2(m)}. \]
Since $u$ is variationally characterized by
\begin{equation}\label{e:upos1}
 \lambda_1(\O)=\int_{\O}A\nabla u\cdot\nabla u\,dx = \min\Big\{\int_{\O}A\nabla\tilde{u}\cdot\nabla\tilde{u}\,dx \ : \ \tilde{u}\in H^1_0(\O), \ \int_{\O}\tilde{u}^2\,dm=1  \Big\}, 
\end{equation}
we have
\[ \int_{\O}A\nabla u\cdot\nabla u\,dx\leq \int_{\O}A\nabla u_+\cdot\nabla u_+\,dx\qquad\text{and}\qquad \int_{\O}A\nabla u\cdot\nabla u\,dx\leq \int_{\O}A\nabla u_-\cdot\nabla u_-\,dx. \]
Then, it follows that the two above inequalities are in fact equalities since otherwise we have
\begin{align*}
\int_{\O}A\nabla u\cdot\nabla u\,dx &=\int_{\O}A\nabla u^+\cdot\nabla u^+\,dx + \int_{\O}A\nabla u^-\cdot\nabla u^-\,dx \\
&>\Big(\int_{\O}(u^+)^2\,dm + \int_{\O}(u^-)^2\,dm \Big)\int_{\O}A\nabla u\cdot\nabla u\,dx = \int_{\O}A\nabla u\cdot\nabla u\,dx,
\end{align*}
which is absurd. In view of the minimization characterization \eqref{e:upos1}, this ensures that $u_+$ is solution of the equation
\[ -\dive(A\nabla u_+)=\lambda_1(\O)u_+b \quad\text{in}\quad \O. \]
Then, the strong maximum principle (see \cite[Theorem 8.19]{gilbarg-trudinger-01}) and the connectedness of $\O$ imply that $u_+$ is strictly positive in $\O$, which completes the proof.
\end{proof}

\subsection{Quasi-minimality and Lipschitz continuity of the eigenfunctions}
We prove that the vector $U=(u_1,\dots,u_k)$ of normalized eigenfunctions on an optimal set $\O^\ast$ for the problem \eqref{e:shapeopt} is a local quasi-minimizer of the vector-valued functional
\[ H^1_0(D,\R^k)\ni \tilde{U}\mapsto \int_DA\nabla\tilde{U}\cdot\nabla\tilde{U}\,dx + \Lambda|\{|\tilde{U}|>0\}| \]
in the sense of the Proposition below. The Lipschitz continuity of the eigenfunctions is then a consequence of Theorem \ref{t:main2}. We notice that, in view of the variational formulation \eqref{e:carsum}, the vector $U$ is solution to the following problem
\begin{equation}\label{e:carU}
\min\bigg\{ \int_DA\nabla V\cdot\nabla V\,dx + \Lambda|\{|V|>0\}| \ : \ V=(v_1,\dots,v_k)\in H^1_0(D,\R^k), \int_Dv_iv_j\,dm=\delta_{ij} \bigg\}.
\end{equation}

\begin{prop}[Quasi-minimality of $U$]\label{p:quasiminU}
Let $\O^\ast\subset D$ be an optimal set for the problem \eqref{e:shapeopt}. Then the vector of orthonormalized eigenfunctions $U=(u_1,\dots,u_k)\in H^1_0(\O^\ast,\R^k)$ satisfies the following quasi-minimality condition:
for every $C_1>0$ there exist constants $\eps\in (0,1)$ and $C>0$, depending only on $d, k, C_1, \|U\|_{L^\infty}$ and $|D|$, such that
\begin{equation}\label{e:quasiminUa}
\int_DA\nabla U\cdot\nabla U\,dx + \Lambda|\{|U|>0\}| \leq \big(1+C\|U-\tilde{U}\|_{L^1}\big) \int_DA\nabla\tilde{U}\cdot\nabla\tilde{U}\,dx + \Lambda|\{|\tilde{U}|>0\}|,
\end{equation}
for every $\tilde{U} \in H^1_0(D,\R^k)$ such that $\|U-\tilde{U}\|_{L^1}\leq\eps$ and $\|\tilde{U}\|_{L^\infty}\leq C_1$.
\end{prop}

The next Lemma, in which we get rid of the orthogonality constraint in \eqref{e:carU}, is similar to Lemma 2.5 in \cite{mazzoleni-terracini-velichkov-17} with only slight modifications, but we decided to recall the whole proof for a sake of completeness.

\begin{lm}\label{l:ortho}
Let $\O\subset D$ be a quasi-open set and let $U=(u_1,\dots,u_k)$ be the vector of normalized eigenvalues on $\O$. Let $\delta>0$. Then there exist $\overline{\eps}_k \in (0,1)$ and $\overline{C}_k>0$, depending only on $d, k, \delta$ and $|\O|$, such that for every $\tilde{U}=(\tilde{u}_1,\dots,\tilde{u}_k) \in H^1_0(D,\R^k)$ satisfying
\[ \eps_k := \sum_{i=1}^k \int_D|\tilde{u}_i-u_i|\,dm \leq \overline{\eps}_k\qquad \text{and} \qquad \sup_{i=1,\dots,k} \Big\{\|u_i\|_{L^\infty} + \|\tilde{u}_i\|_{L^\infty}\Big\} \leq \delta \]
the following estimate holds
\begin{equation}\label{e:ortho}
\int_DA\nabla V\cdot\nabla V\,dx \leq (1+\overline{C}_k\eps_k)\int_DA\nabla\tilde{U}\cdot\nabla\tilde{U}\,dx,
\end{equation}
where $V=(v_1,\dots,v_k) \in H^1_0(D,\R^k)$ is the vector obtained by orthonormalizing $\tilde{U}$ with the Gram-Schmidt procedure: 
\[ v_i=w_i/\|w_i\|_{L^2(m)} \qquad \text{where} \quad 
      w_i=\left\{
          \begin{aligned}
            &\tilde{u}_1 &\text{if} \quad& i=1\\
            &\tilde{u}_i-\textstyle\sum_{j=1}^{i-1}\left(\int_D \tilde{u}_iv_j\,dm\right)v_j &\text{if}\quad& i=2,\dots,k.\\
          \end{aligned}
        \right.
     \]
\end{lm}

\begin{proof}
We first prove an estimate of $\|u_k-w_k\|_{L^2(m)}$ in terms of $\eps_k$. Precisely, we prove by induction on $k$ that there exist constants $\overline{\eps}_k\in(0,1)$ and $C_k>0$ such that the following estimates hold whenever $\eps_k\leq\overline{\eps}_k$
\begin{equation}\label{e:ortho3}
\|u_k-w_k\|_{L^1(m)} \leq C_k\eps_k, \quad \sum_{i=1}^k\|u_i-v_i\|_{L^1(m)}\leq C_k\eps_k, \quad\max_{i=1,\dots,k}\|v_i\|_{L^\infty}\leq C_k.
\end{equation}
For $k=1$ the first estimate obviously holds with $C_1\geq 1$. Moreover we have
\begin{align}\label{e:ortho1}
\|\tilde{u}_1-v_1\|_{L^1(m)} 
\nonumber &=\frac{|\|\tilde{u}_1\|_{L^2(m)}-1|}{\|\tilde{u}_1\|_{L^2(m)}}\|\tilde{u}_1\|_{L^1(m)} 
\leq \frac{|\|\tilde{u}_1\|_{L^2(m)}^2-1|}{\|\tilde{u}_1\|_{L^2(m)}^2}\|\tilde{u}_1\|_{L^1(m)} \\
&=\frac{|\|u_1+(\tilde{u}_1-u_1)\|_{L^2(m)}^2-1|}{\|u_1+(\tilde{u}_1-u_1)\|_{L^2(m)}^2}\|u_1+(\tilde{u}_1-u_1)\|_{L^1(m)} \\
\nonumber &\leq \frac{2\int u_1|\tilde{u}_1-u_1|\,dm + \|\tilde{u}_1-u_1\|_{L^2(m)}^2}{1-2\int u_1|\tilde{u}_1-u_1|\,dm}\Big(\|u_1\|_{L^1(m)}+\|\tilde{u}_1-u_1\|_{L^1(m)} \Big) \\
\nonumber&\leq \frac{(2\|u_1\|_{L^\infty}+\|\tilde{u}_1-u_1\|_{L^\infty})\|\tilde{u}_1-u_1\|_{L^1(m)}}{1-2\|u_1\|_{L^\infty}\|\tilde{u}_1-u_1\|_{L^1(m)}}\Big(\|u_1\|_{L^1(m)}+\|\tilde{u}_1-u_1\|_{L^1(m)} \Big) \\
\nonumber&\leq \frac{3\delta\eps_1}{1-2\delta\eps_1}\Big(|\O|^{\sfrac12}+\eps_1 \Big)
\leq 12\delta|\O|^{\sfrac12}\eps_1,
\end{align}
where the last inequality holds if $\eps_1\leq\min\{(4\delta)^{-1},|\O|^{\sfrac12}\}$. This gives the following $L^1$-estimate
\[ \|u_1-v_1\|_{L^1(m)} \leq \|u_1-\tilde{u}_1\|_{L^1(m)} + \|\tilde{u}_1-v_1\|_{L^1(m)} \leq (1+12\delta|\O|^{\sfrac12})\eps_1. \]
Finally, we estimate the infinity norm
\begin{align*}
\|v_1\|_{L^\infty} &= \frac{\|\tilde{u}\|_{L^\infty}}{\|\tilde{u}_1\|_{L^2(m)}}
= \frac{\|\tilde{u}\|_{L^\infty}}{\|u_1+(\tilde{u}_1-u_1)\|_{L^2(m)}} 
\leq \frac{\|\tilde{u}\|_{L^\infty}}{(1-2\int u_1|\tilde{u}_1-u_1|\,dm)^{\sfrac12}} \\
&\leq \frac{\|\tilde{u}\|_{L^\infty}}{1-2\int u_1|\tilde{u}_1-u_1|\,dm}
\leq \frac{\delta}{1-2\delta\eps_1} \leq 2\delta,
\end{align*}
which proves the claim for $k=1$.
Suppose now that the claim holds for $1,\dots,k-1$. We first estimate $\|u_k-w_k\|_{L^1(m)}$. Since the functions $u_i$ form an orthogonal system of $L^2(\O,m)$ and by the induction's hypothesis we have (and also because $\eps_{k-1}\leq\eps_k$)
\begin{align}\label{e:ortho2}
\sum_{i=1}^{k-1}\bigg|\int_D\tilde{u}_k v_i\,dm\bigg| \nonumber&\leq \sum_{i=1}^{k-1}\int_D\Big( |\tilde{u}_k-u_k|u_i+|v_i-u_i|u_k+|v_i-u_i||\tilde{u}_k-u_k|\Big)\,dm \\
&\leq ((k-1)\delta + \delta C_{k-1} + (k-1)(C_{k-1} + \delta))\eps_k =: \tilde{C}_k \eps_k.
\end{align}
Therefore, with the triangle inequality we obtain
\begin{align}\label{e:ortho4}
\nonumber\|u_k-w_k\|_{L^1(m)} &\leq \|u_k-\tilde{u}_k\|_{L^1(m)} + \sum_{i=1}^{k-1}\bigg|\int_D\tilde{u}_k v_i\,dm\bigg| \,(\|u_i\|_{L^1(m)} + \|v_i-u_i\|_{L^1(m)}) \\
&\leq (1+\tilde{C}_k(|\O|^{\sfrac12}+\eps_{k-1}))\eps_k\leq (1+2\tilde{C}_k|\O|^{\sfrac12})\eps_k.
\end{align}
We now prove the second estimate in \eqref{e:ortho3}.
Using once again \eqref{e:ortho2}, we have the following estimate of the $L^\infty$-norm of $w_k$
\begin{equation}\label{e:ortho5}
\|w_k\|_{L^\infty} \leq \|\tilde{u}_k\|_{L^\infty}+\sum_{i=1}^k\bigg|\int_D\tilde{u}_kv_i\,dm\bigg|\|v_i\|_{L^\infty}
\leq \delta + C_{k-1}\tilde{C}_k.
\end{equation}
Moreover, with the same procedure as in \eqref{e:ortho1}, it follows from \eqref{e:ortho4} and \eqref{e:ortho5} that
\begin{align*}
\|\tilde{u}_k-v_k\|_{L^1(m)} &\leq \frac{(3\|u_k\|_{L^\infty}+\|w_k\|_{L^\infty})\|w_k-u_k\|_{L^1(m)}}{1-2\|u_k\|_{L^\infty}\|w_k-u_k\|_{L^1(m)}}\Big(\|u_k\|_{L^1(m)}+\|w_k-u_k\|_{L^1(m)} \Big) \\
&\leq \frac{(4\delta+C_{k-1}\tilde{C}_k)(1+2\tilde{C}_k|\O|^{\sfrac12})}{1-2\delta(1+2\tilde{C}_k|\O|^{\sfrac12})\eps_k}(|\O|^{\sfrac12} + (1+2\tilde{C}_k|\O|^{\sfrac12}))\eps_k.
\end{align*}
Now, choose $\eps_k \leq [4\delta (1+2\tilde{C}_k|\O|^{\sfrac12})]^{-1}$ so that with the triangle inequality we have
\begin{align*}
\|u_k-v_k\|_{L^1(m)} &\leq \|u_k-\tilde{u}_k\|_{L^1(m)} + \|\tilde{u}_k-v_k\|_{L^1(m)} \\
&\leq \Big[1+2(4\delta+C_{k-1}\tilde{C}_k)(1+2\tilde{C}_k|\O|^{\sfrac12})(|\O|^{\sfrac12} + (1+2\tilde{C}_k|\O|^{\sfrac12})\Big]\eps_k.
\end{align*}
We then use the inductive hypothesis to get the desired $L^1$-estimate. It remains only to estimate $\|v_k\|_{L^\infty}$. Firstly, notice that we have
\begin{align*}
|\ \|w_k\|_{L^2(m)}-1| &\leq |\ \|w_k\|_{L^2(m)}^2-1| = |\ \|u_k + (w_k-u_k)\|_{L^2(m)}^2-1| \\
&\leq \bigg| 2\int_D u_k(u_k-w_k)\,dm + \int_D(u_k-w_k)^2\,dm\bigg| \\
&\leq (3\|u_k\|_{L^\infty}  + \|w_k\|_{L^\infty})\|u_k-w_k\|_{L^1(m)} \\
&\leq (4\delta + C_{k-1}\tilde{C}_k)(1+2\tilde{C}_k|\O|^{\sfrac12})\eps_k.
\end{align*}
Thus, with the extra assumption $\eps_k \leq [(4\delta + C_{k-1}\tilde{C}_k)(1+2\tilde{C}_k|\O|^{\sfrac12})]^{-1}$, it follows that $1/2 \leq \|w_k\|_{L^2(m)} \leq 3/2$. With \eqref{e:ortho5} this gives the following $L^\infty$-norm of $v_k$
\[ \|v_k\|_{L^\infty} = \frac{\|w_k\|_{L^\infty}}{\|w_k\|_{L^2(m)}} \leq 2(\delta + C_{k-1}\tilde{C}_k) \]
and concludes the proof of the claim.

We are now in position to prove the Lemma by induction. For $k=1$, we ask that $\eps_1\leq(4\delta)^{-1}$, so that we have
\begin{align*}
\int_DA\nabla v_1\cdot\nabla v_1\,dx &\leq \|\tilde{u}_1\|_{L^2(m)}^{-2}\int_DA\nabla \tilde{u}_1\cdot\nabla \tilde{u}_1\,dx \\
&\leq (1-2\|u_1\|_{L^\infty}\|\tilde{u}_1-u_1\|_{L^1(m)})^{-2}\int_DA\nabla \tilde{u}_1\cdot\nabla \tilde{u}_1\,dx \\
&\leq (1+4\delta\eps_1)^2\int_DA\nabla \tilde{u}_1\cdot\nabla \tilde{u}_1\,dx 
\leq (1+12\delta\eps_1)\int_DA\nabla \tilde{u}_1\cdot\nabla \tilde{u}_1\,dx.
\end{align*}
Suppose now that the Lemma holds for $1,\dots,k-1$. Thanks to the first estimate in \eqref{e:ortho3} of the preceding claim we have
\[ \|w_k\|_{L^2(m)}^{-2} \leq (1-2\|u_k\|_{L^\infty}\|u_k-w_k\|_{L^1(m)})^{-2} \leq (1+4\delta C_k\eps_k)^2 \leq 1+12\delta C_k\eps_k, \]
where the last inequality holds if $\eps_k \leq (4\delta C_k)^{-1}$.
On the other hand, for every $i=1,\dots,k-1$, we have by the inductive assumption
\[ \int_DA\nabla v_i\cdot\nabla v_i\,dx \leq \sum_{j=1}^{k-1}\int_DA\nabla v_j\cdot\nabla v_j\,dx \leq (1+\overline{C}_{k-1}\eps_{k-1})\sum_{j=1}^{k-1}\int_DA\nabla\tilde{u}_j\cdot\nabla\tilde{u}_j\,dx. \]
Therefore, using the estimate \eqref{e:ortho2} we get
\begin{multline*}
\left(\int_DA\nabla w_k\cdot\nabla w_k\,dx\right)^{1/2} \leq \left(\int_DA\nabla \tilde{u}_k\cdot\nabla \tilde{u}_k\,dx\right)^{1/2} + \sum_{i=1}^{k-1}\bigg|\int_D\tilde{u}_kv_i\,dm\bigg|\left(\int_DA\nabla v_i\cdot\nabla v_i\,dx\right)^{1/2} \\
\leq \left(\int_DA\nabla \tilde{u}_k\cdot\nabla \tilde{u}_k\,dx\right)^{1/2} + \tilde{C}_k\eps_k(1+\overline{C}_{k-1})^{1/2} \left(\sum_{j=1}^{k-1}\int_DA\nabla\tilde{u}_j\cdot\nabla\tilde{u}_j\,dx\right)^{1/2}.
\end{multline*}
We then ask that $\eps_k \leq (2\tilde{C}_k)^{-1}(1+\overline{C}_{k-1})^{-1/2}$ so that we get
\begin{multline*}
\int_DA\nabla v_k\cdot\nabla v_k\,dx = \|w_k\|_{L^2(m)}^{-2}\int_DA\nabla w_k\cdot\nabla w_k\,dx \\
\leq (1+12\delta C_k\eps_k)\left((1+\tilde{C}_k(1+\overline{C}_{k-1})^{1/2}\eps_k)\int_DA\nabla \tilde{u}_k\cdot\nabla \tilde{u}_k\,dx + \sum_{j=1}^{k-1}\int_DA\nabla\tilde{u}_j\cdot\nabla\tilde{u}_j\,dx\right).
\end{multline*}
This, using once again the inductive hypothesis, proves \eqref{e:ortho} and concludes the proof.
\end{proof}

\begin{proof}[Proof of Proposition \ref{p:quasiminU}]
Let  $\tilde{U}$ be a vector satisfying the hypothesis of Proposition \ref{p:quasiminU} and let $V\in H^1_0(D,\R^k)$ be the vector given by Lemma \ref{l:ortho} and obtained by orthonormalizing $\tilde{U}$. Using $V$ as a test function in \eqref{e:carU} and then Lemma \ref{l:ortho}, we have
\begin{align*}
\int_DA\nabla U\cdot\nabla U\,dx + \Lambda|\{|U|>0\}| &\leq \int_DA\nabla V\cdot\nabla V\,dx + \Lambda|\{|V|>0\}| \\
& \leq (1+\overline{C}_k\|U-\tilde{U}\|_{L^1})\int_DA\nabla\tilde{U}\cdot\nabla\tilde{U}\,dx + \Lambda|\{|\tilde{U}|>0\}|,
\end{align*}
where we have used in the last inequality that $\{|V|>0\} \subset \{|\tilde{U}|>0\}$ (which holds by construction of $V$).
\end{proof}

We now conclude the proof of Theorem \ref{t:main1}.

\begin{proof}[Proof of Theorem \ref{t:main1}]
The proof follows from Proposition \ref{p:quasiminU} and Theorem \ref{t:main2} (see also Lemma \ref{l:boundedness}).
\end{proof}

%

\bigskip\bigskip
\noindent {\bf Acknowledgments.} 
This work was partially supported by the French Agence Nationale de la Recherche (ANR) with the projects GeoSpec (LabEx PERSYVAL-Lab, ANR-11-LABX-0025-01) and the project SHAPO (ANR-18-CE40-0013).

\bibliographystyle{plain}
\bibliography{bib-t}

\end{document}